\documentclass[a4paper,10pt]{amsart}
\usepackage[colorlinks,linkcolor=blue,citecolor=blue]{hyperref}
\usepackage{latexsym, amssymb,amsmath,amsthm}
\usepackage{mathrsfs}
\usepackage{bbm}

\usepackage[all, knot]{xy}
\xyoption{arc}
\def \To{\longrightarrow}
\def \dim{\operatorname{dim}}

\def \rad{\operatorname{rad}}

\def \Vec{\operatorname{Vec}}

\def \mod{\operatorname{mod}}

\def \Z{\mathbb{Z}}

\def \k{\mathbbm{k}}
\def \1{\mathbf{1}}

\def\A{\mathcal{A}}
\def\g{\mathfrak{g}}

\numberwithin{equation}{section}

\newtheorem{theorem}{Theorem}[section]
\newtheorem{lemma}[theorem]{Lemma}
\newtheorem{proposition}[theorem]{Proposition}
\newtheorem{corollary}[theorem]{Corollary}
\newtheorem{definition}[theorem]{Definition}
\newtheorem{example}[theorem]{Example}
\newtheorem{remark}[theorem]{Remark}

\begin{document}

\title{Finite-dimensional quasi-Hopf algebras of Cartan type}



\author{Yuping Yang}
\address{Yang: School of Mathematics and statistics, Southwest University, Chongqing 400715, China}
\email{yupingyang@swu.edu.cn}

\author{Yinhuo zhang}
\address{Zhang: Department of Mathematics $\&$ Statistics, University of Hasselt, Universitaire Campus, 3590 Diepeenbeek, Belgium}
\email{yinhuo.zhang@uhasselt.be}

\date{}
\maketitle

\begin{abstract}
In this paper, we present a general method for constructing finite-dimensional quasi-Hopf algebras from finite abelian groups and braided vector spaces of Cartan type. The study of such quasi-Hopf algebras leads to the classification of finite-dimensional radically graded basic quasi-Hopf algebras over abelian groups  with dimensions  not divisible by $2,3,5,7$ and  associators  given by abelian $3$-cocycles. As special cases , the small quasi-quantum groups are introduced and studied. Many new explicit examples of finite-dimensional genuine quasi-Hopf algebras are obtained.
\end{abstract}

\section{introduction}
Quasi-Hopf algebras are generalizations of Hopf algebras, and are  fundamental in the study of finite integral tensor categories \cite{EGNO}. Recall that a tensor category is called integral if the Frobinus-Perron dimension of each object is an integer.  According to \cite{EO}, any finite integral tensor category over an algebraically closed filed is equivalent to the representation category of some finite-dimensional quasi-Hopf algebra. Pointed tensor categories are special examples of integral tensor categories, and the corresponding quasi-Hopf algebras are called basic quasi-Hopf algebras.

In the past and half decades, the classification of finite-dimensional basic quasi-Hopf algebras have attracted lots of attention. Since the dual of a finite-dimensional pointed Hopf algebra is a basic Hopf algebra,  the duals of the finite dimensional pointed Hopf algebras  over abelian groups classified in\cite{AS,AS2,H4,AG} provide a big family of finite-dimensional basic quasi-Hopf algebras. Since our ultimate goal is to classify the tensor categories, we are only interested in those quasi-Hopf algebras whose representation categories do not arise from  any Hopf algebra. Such quasi-Hopf algebras are said to be genuine.  In \cite{EG1,EG2,EG3}, Etingof and Gelaki gave a method for constructing basic genuine quasi-Hopf algebras from known basic Hopf algebras, and classified the finite-dimensional radically graded basic quasi-Hopf algebras over cyclic groups of prime order. In \cite{G}, Gelaki constructed the finite-dimensional basic quasi-Hopf algebras of dimension $N^3$ over cyclic groups of order $N$. Utilizing  the classification result of \cite{AS2}, Angiono \cite{A} classified the finite-dimensional radically graded basic quasi-Hopf algebras over cyclic groups with dimensions not divisible by small prime divisors. In \cite{HLYY,HY0,HY}, the quasi-commutative  finite-dimensional graded pointed Majid algebras of low ranks (dual basic quasi-Hopf algebras) are classified by the first author and his cooperators.  Although the aforementioned classification work covered a lot of new finite dimensional quasi-Hopf algebras,  the most important family of  finite dimensional pointed Hopf algebras of Cartan type is not yet covered by the above classifications of quasi-Hopf algebras.  In particular,  we have not found a natural quasi-version of the (generalized)  small quantum groups although a very close quasi-version of the  Frobenius-Lusztig kernel was constructed by means of  the quasi-quantum double in \cite{LOZ}.  This is because the classic construction of a small quantum group as a particular quotient of  the quantum double works not for the quasi-Hopf algebra case. So we have to look for an alternative way to define the notion of a small quasi-quantum group. The fact that the classical small quantum groups form a special class of the finite dimensional pointed Hopf algebras of finite Cartan type, see  \cite{L,L1,AS,AS2},  inspires us: if we could construct finite dimensional quasi-Hopf algebras from Cartan matrices, then the small quasi-quantum groups must be the special cases of those quasi-Hopf algebras of Cartan type.  This motivates us to study the finite-dimensional quasi-Hopf algebras of Cartan type. The main work of this paper are threefold.

First of all, we present a general method for constructing finite-dimensional quasi-Hopf algebras from finite Cartan matrices. Such a quasi-Hopf algebra is generated by an abelian group and a braided vector space of Cartan type. In more detail, let $G$ be a finite abelian group and $\mathbbm{G}$ a bigger abelian group uniquely determined by $G$, see \eqref{e3.4}. Let $u(\mathfrak{D},\lambda,\mu)$ (\cite{AS2}, or see Theorem \ref{T2.13}) be the generalized small quantum group generated by grouplike elements $\mathbbm{G}$ and skew-primitive elements $\{X_1,\cdots, X_n\}$.  We then  determine the $2$-cochains $J$ on $\mathbbm{G}$ such that the subalgebra of the twist quasi-Hopf algebra $u(\mathfrak{D},\lambda,\mu)^J$ generated by $G$ and $\{X_1,\cdots, X_n\}$ is a quasi-Hopf subalgebra, denoted  $u(\mathfrak{D},\lambda,\mu,\Phi_J)$, see Theorem \ref{T3.4}. Note that if $\lambda=0$ and $\mu=0$, then $u(\mathfrak{D},0, 0)$ is  a radically graded basic Hopf algebra. Moreover, when $G$ is a cyclic group, the quasi-Hopf algebra $u(\mathfrak{D},0,0,\Phi_J)$ is the same as those constructed in \cite{A,EG1,EG2,EG3}.  However, if  $G$ is not cyclic, or $u(\mathfrak{D},\lambda,\mu)$ is not radically graded and basic, then the construction and the study of $u(\mathfrak{D},\lambda,\mu,\Phi_J)$ are much more complicated. One of the difficulties is to compute suitable $2$-cochain $J$'s on $\mathbbm{G}$ such that  $u(\mathfrak{D},\lambda,\mu,\Phi_J)$ is a quasi-Hopf algebra.  Even if we can compute such a suitable cochain $J$,  we still have no standard method to  determine whether $u(\mathfrak{D},\lambda,\mu,\Phi_J)$ is genuine or not. While in the case of $\lambda=0, \mu=0$ and $G$ is a cyclic group, this problem is trivial.

Secondly, the obtained quasi-Hopf algebras of Cartan type deliver the classification of finite-dimensional radically graded basic quasi-Hopf algebras over abelian groups.  Let $H$ be a finite-dimensional basic quasi-Hopf algebra,  and $\rad(H)$ the Jacobson radical of $H$. Then we have $H/\rad(H)\cong [\k G]^*$, where $G$ is the Grothendieck group of the representation category of $H$.  We say that the basic quasi-Hopf algebra $H$ is over the group $G$. When $G$ is abelian, it is obvious that $H/\rad(H)\cong \k G$.   If  $H$ is a radically graded and  basic quasi-Hopf algebra over $G$, then the associator of $H$ is determined by a normalized $3$-cocycle on $G$, see \cite{A,HLYY,HY}. We show that a finite-dimensional radically graded and basic quasi-Hopf algebra $H$ over an abelian group $G$ with dimension not divisible by $2,3,5,7$, and the associator is given by an abelian $3$-cocycle of $G$, must be isomorphic to a quasi-Hopf algebra of Cartan type $u(\mathfrak{D},\lambda,\mu,\Phi_J)$, where $\lambda=0, \mu=0$, see Theorem \ref{classify}. Since each normalized $3$-cocycle of a cyclic group or an abelian group of the form $\Z_m\times \Z_n$  is abelian,  our classification extends  the corresponding classification results of \cite{A,HLYY} to more general cases.

Thirdly,  we introduce the quasi-version of  the small quantum groups, which form a class of finite dimensional quasi-Hopf algebras of Cartan type, namely,  those $H_{\underline{c}},$ where $\underline{c}$ is a family of parameters. When  $\underline{c}$ approaches $0$,   the small quasi-quantum group $H_{\underline{c}}$ will be the usual  small quantum group, see Proposition \ref{P5.3}.  As mentioned before, the small quasi-quantum group  defined in this paper is  substantially different from the one defined in \cite{LOZ}, where a small quasi-quantum group is defined as the quantum double of a quasi-Hopf algebra $A_q(\g)$ constructed in \cite{EG2}, where $\g$ is a simple Lie algebra.  Note that  the quantum double $D(A_q(\g))$ is a quasi-Hopf algebra of Cartan type as well.  Unlike the Hopf algebra case,  the small quasi-quantum group $H_{\underline{c}}$ is not a quotient of the double $D(A_q(\g))$ in general.  For example,  if  the order of $q$ is odd and not divisible by $3$ in case $\g$ is of type $G_2$, then the double  $D(A_q(\g))$ is not a genuine quasi-Hopf algebra, see \cite{EG4}. Under the same conditions for $q$, we can show that there are many genuine small quasi-quantum groups. This means that those small quasi-quantum groups can not be the quotients of $D(A_q(\g))$.

Beside the study of the small quasi-quantum groups, we will provide lots of other genuine quasi-Hopf algebras associated to  finite Cartan matrices in Section 6. As a matter of fact, our method will not only systematically produce  many nonsemisimple, nonradically graded genuine quasi-Hopf algebras, but also  yield  many new classes of  finite integral  and non-pointed tensor categories.

The paper is organized as follows. In Section 2, we introduce some concepts and known results about quasi-Hopf algebras, generalized small quantum groups and  $3$-cocycle of abelian groups. In Section 3, the quasi-Hopf algebras of Cartan type are constructed, and some low rank nonradically graded examples are provided. In Section 4, we  classify the
finite-dimensional radically graded quasi-Hopf algebras which are basic over abelian groups, and show that all the radically graded quasi-Hopf algebras of Cartan type are genuine. In Section 5, we introduce the small quasi-quantum groups, which are special nonradically graded quasi-Hopf algebras of Cartan type, and present explicitly  examples of   genuine quasi-Hopf algebras of Cartan type. Section 6  is devoted to new examples of nonradically graded genuine quasi-Hopf algebras associated to connected finite Cartan matrices.

Throughout this paper, $\k$ denotes an algebraically closed field of characteristic zero. All the algebras, tensor categories and the unadorned tensor product $\otimes$ are over $\k$.

\section{Preliminaries}
In this section, we introduce some notations and basic facts about Quais-Hopf algebras, tensor categories and some important results \cite{AS2} about pointed Hopf algebras.
\subsection{Quasi-Hopf algebras}
A qausi-bialgebra $H=(H,\bigtriangleup,\varepsilon, \Phi)$ is an unital associative algebra with two algebra maps $\bigtriangleup:H\to H\otimes H$(the comultiplication) and $\varepsilon:H\to \k$ (the counit), and an invertible
element $\Phi\in H^{\otimes 3}$ (the associator), subject to:
\begin{eqnarray*}
&(\varepsilon\otimes id)\bigtriangleup(h)=h=(id\otimes \varepsilon)\bigtriangleup(h),&\\
&(id\otimes \bigtriangleup)\bigtriangleup(h)=\Phi\cdot(\bigtriangleup\otimes id)\bigtriangleup(h)\cdot\Phi^{-1},&\\
&(id\otimes id\otimes \bigtriangleup)(\Phi)\cdot(\bigtriangleup\otimes id\otimes id)(\Phi)=
(1\otimes \Phi)\cdot(id\otimes \bigtriangleup\otimes id)(\Phi)\cdot(\Phi\otimes 1),&\\
&(id\otimes \varepsilon\otimes id)(\Phi)=1&
\end{eqnarray*}
for all $h\in H$. Write $\Phi=\Phi^1\otimes \Phi^2\otimes \Phi^3$ and $\Phi^{-1}=\overline{\Phi}^1\otimes \overline{\Phi}^2\otimes \overline{\Phi}^3.$
A quasi-Hopf algebra $H=(H,\bigtriangleup,\varepsilon, \Phi, S,\alpha,\beta)$ is a quasi-bialgebra $(H,\bigtriangleup,\varepsilon, \Phi)$ with an antipode $(S,\alpha,\beta),$ where $\alpha,\beta\in H$ and $S:H\to H$ is an angebra anti-homomorphism satisfying
\begin{eqnarray*}
&\sum S(a_1)\alpha a_2=\varepsilon(a)\alpha,\  \sum a_1\beta S(a_2)=\varepsilon(a)\beta,&\\
&\Phi^1\beta S(\Phi^2)\alpha\Phi^3=1,\ S(\overline{\Phi}^1)\alpha \overline{\Phi}^2 \beta S(\overline{\Phi}^3)=1&
\end{eqnarray*}
 for all $a\in H.$ Here we use Sweedler's notation $\bigtriangleup(a)=\sum a_1\otimes a_2.$
\begin{definition}\label{D2.1}
A twist for a quasi-Hopf algebra $H$ is an invertible element $J\in H\otimes H$ satisfying
$$(\varepsilon\otimes id)(J)=(id\otimes \varepsilon)(J)=1.$$
\end{definition}
Suppose that $J=\sum_if_i\otimes h_i$ is a twist of $H$ with inverse $J^{-1}=\sum_i\overline{f_i} \otimes \overline{h_i}.$  Write
\begin{equation}\label{2.1}
\alpha_J=\sum_iS(\overline{f_i})\alpha\overline{g_i}, \ \ \beta_J=\sum_if_i\beta S(g_i).
\end{equation}
According to \cite{Dr}, if $\beta_J$ is invertible then one can define a new quasi-Hopf algebra structure $H^J=(H,\bigtriangleup_J,\varepsilon,\Phi_J,S_J,\beta_J\alpha_J,1)$ on the algebra $H,$ where
\begin{eqnarray}
&&\bigtriangleup_J(h)=J\bigtriangleup(h)J^{-1},\ \ h\in H,\label{2.2}\\
&&\Phi_J=(1\otimes J)(id\otimes \bigtriangleup)(J)\Phi(\bigtriangleup\otimes id)(J^{-1})(J\otimes 1)^{-1}\label{2.3}\\
&&S_J(h)=\beta_JS(h){\beta_{J}}^{-1}, \ \ h\in H.\label{2.4}
\end{eqnarray}
Two quasi-Hopf algebras $H$ and $H'$ are said to be twist equivalent if $H'\cong H^J$ for some twist $J$ of $H.$
\begin{definition}\label{D2.2}
A quasi-Hopf algebra $H$ is genuine if $H$ is not twist (or gauge) equivalent to any Hopf algebra.
\end{definition}
The following theorem is useful in Section 5.
\begin{theorem}\cite[Theorem 2.2]{NS}\label{T2.3}
Let $H$ and $B$ be two finite-dimensional quasi-Hopf algebras. Then the two module categories $H$-mod and $B$-mod are tensor equivalent if and only if $H$ is equal to $B^J$ for some twist $J$ of $B.$
\end{theorem}

Let $H=(H,\bigtriangleup,\varepsilon, \Phi, S,\alpha,\beta)$ be a quasi-Hopf algebra. If $H=\oplus_{i\geq 0}H[i]$ is a graded algebra such that $(H[0],\varepsilon, \Phi, S,\alpha,\beta)$ is a quasi-Hopf subalgebra, and $I=\oplus_{i\geq 1}H[i]$ is the Jacobson radical of $H$ and $I^k=\oplus_{i\geq k}H[i]$ for each $k\geq 1,$ then we call $H$ a radically graded quasi-Hopf algebra.
Suppose that $H=(H,\bigtriangleup,\varepsilon, \Phi, S,\alpha,\beta)$ is a quasi-Hopf algebra, $I$ is the Jacobson radical of $H.$ If $I$ is a quasi-Hopf ideal of $H$, i.e., $\bigtriangleup(I)\subset H\otimes I+I\otimes H$, $S(I)=I$ and $\varepsilon(I)=0,$ then we can construct a radically graded quasi-Hopf algebra associated to $H.$ Let $H[0]=H/I$ and $\pi:H\to H[0]$ is the canonical projection. Define $H[k]=I^k/I^{k+1}$ for $k\geq 1.$ then the graded algebra $gr(H)=\oplus_{i\geq 0}H[i]$ has a natural quasi-Hopf algebra structure, with the associator $\pi\otimes \pi\otimes \pi(\Phi)$, and the antipode $(\pi\circ S,\pi(\alpha),\pi(\beta)).$ For radically graded quasi-Hopf algebras, we have the following useful lemma.

\begin{lemma}\cite[Lemma 2.1]{EG1}\label{L2.4}
Let $H=\oplus_{i\geq 0}H_{i}$ be a radically graded quasi-Hopf algebra.  Then $H$ is generated by $H[0]$ and $H[1].$
\end{lemma}

\subsection{Datum of Cartan type, root system and Wyle group}
For a finite group $G,$ by $\widehat{G}$ we mean the character group of $G.$  We give the definition of a datum of Cartan type according to \cite{AS2}.
\begin{definition}\label{D2.4}
A datum of Cartan type
\begin{equation}
\mathfrak{D}=\mathfrak{D}(G,(h_i)_{1\leq i\leq \theta},(\chi_i)_{1\leq i\leq \theta},A)
\end{equation}
consists of an abelian group $G,$ elements $h_i\in G,$  characters $\chi_i\in \widehat{G},$ $1\leq i\leq \theta,$ and a generalized Cartan matrix $A=(a_{ij})_{1\leq i,j\leq \theta}$ of size $\theta$ satisfying
\begin{equation}
q_{ij}q_{ji}=q_{ii}^{a_{ij}},\ \mathrm{where}\ q_{ij}=\chi_j(h_i), \mathrm{for\ all}\ 1\leq i,j\leq \theta.
\end{equation} We call $\theta$ the rank of $\mathfrak{D}.$ A datum of Cartan type $\mathfrak{D}$ is called finite Cartan type if the associated Cartan matrix $A$ is finite; $\mathfrak{D}$ is said to be connected if $A$ is a connected Cartan matrix.
\end{definition}

 Fix a datum $\mathfrak{D}=\mathfrak{D}(G,(h_i)_{1\leq i\leq n},(\chi_i)_{1\leq i\leq \theta},A)$ of finite Cartan type. Let $\{\alpha_i|1\leq i\leq \theta\}$ be the set of free generators of $\Z^\theta$ and $s_i:\Z^\theta\to \Z^\theta$ the reflection $s_i(\alpha_j)=\alpha_j-a_{ij}\alpha_i$ for $1\leq i,j\leq \theta.$  The Weyl group $W$ of $A$ is generated by $\{s_i|1\leq i\leq \theta\}$ and the root system $R=\cup_{i=1}^\theta W(\alpha_i).$ Let $R^+$ be the set of positive roots with respect to the simple roots $\{\alpha_i|1\leq i\leq \theta\}.$ For each $\alpha=\sum_{i=1}^{\theta}k_i\alpha_i\in \Z^\theta,$ denote by $ht(\alpha)=\sum_{i=1}^\theta k_i,$
the height of $\alpha,$ and
\begin{eqnarray}
 &&h_{\alpha}=h_1^{k_1}h_2^{k_2}\cdots h_\theta^{k_\theta},\label{e2.7}\\
 &&\chi_\alpha=\chi_1^{k_1}\chi_2^{k_2}\cdots \chi_\theta^{k_\theta}.
\end{eqnarray} If $\alpha=\sum_{i=1}^{\theta}k_i\alpha_i\in R^+,$ it is obvious that $k_i\geq 0,\ 1\leq i\leq \theta,$ and $ht(\alpha)>0.$
For a datum of Cartan type $\mathfrak{D},$ we can define a Yetter-Drinfeld module $V(\mathfrak{D})$ in $_G^G \mathcal{YD}$ by
\begin{equation}
V(\mathfrak{D})=\sum_{i=1}^\theta V_{h_i}^{\chi_i}
\end{equation} where $V_{h_i}^{\chi_i}=\k\{X_i\}$ is the $1$-dimensional Yetter-Drinfeld module such that the module and the comodule structures are given by
\begin{equation}\label{2.10}
\delta(X_i)=X_i\otimes h_i,\ g\triangleright X_i=\chi_i(h)X_i
\end{equation} for all $g\in G.$ A basis $\{X_1,\cdots,X_\theta\}$ of Yetter-Drinfeld module $V(\mathfrak{D})$ satisfying \eqref{2.10} is called a canonical basis.
It is well-known that $_G^G \mathcal{YD}$ is a braided tensor category.  The natural braiding on  $V(\mathfrak{D})$ is  given by
\begin{equation}
c_{V,V}:V\otimes V\To V\otimes V,\  X_i\otimes X_j\to q_{ij}X_j\otimes X_i,\ 1\leq i,j\leq \theta.
\end{equation}

\subsection{Braided Hopf algebras}\label{s2.3}
Let $(V,c)$ be a braided vector space with a basis $\{X_1,\cdots,X_n\}$ such that
$$c(X_i\otimes X_j)=q_{ij}X_j\otimes X_i, q_{ij}\in \k, 1\leq i,j\leq n.$$
Then we call $(V,c)$ a braided vector space of diagonal type, $\{X_1,\cdots, X_n\}$ a canonical basis of $V,$ and $(q_{ij})_{1\leq i,j\leq n}$ the braiding constants of $V.$ Moreover, if $$q_{ij}q_{ji}=q_{ii}^{a_{ij}}=q_{jj}^{a_{ji}},1\leq i,j\leq n,$$ where $A=(a_{ij})_{1\leq i,j\leq n}$ is a Cartan matrix, then $(V,c)$ is called {\bf{braided vector space of Cartan type}}. For a datum of Cartan type $\mathfrak{D},$ it is obvious that $V(\mathfrak{D})$ is a braided vector space of Cartan type.

Note that the braiding matrix $(q_{ij})_{1\leq i,j\leq \theta}$ of $(V,c)$ defines a braided commutator on $T(V)$ as follows:
\begin{equation}
[X,Y]_c=XY-(\prod_{1\leq k\leq s,1\leq l\leq t}q_{i_kj_l}^{x_ky_l})YX,
\end{equation}
where $X=X_{i_1}^{x_1}X_{i_2}^{x_2}\cdots X_{i_s}^{x_s}$ and $Y=Y_{j_1}^{y_1}Y_{j_2}^{y_2}\cdots Y_{j_s}^{y_s}.$ The braided adjoint action of an element $X\in T(V)$ is defined by
\begin{equation}
ad_c(X)(Y)=[X,Y]_c
\end{equation} for any $Y\in T(V).$

In the rest of this subsection,  we let $\mathfrak{D}=\mathfrak{D}(G,(h_i)_{1\leq i\leq n},(\chi_i)_{1\leq i\leq \theta},A)$ be a connected datum of finite Cartan type. In addition, we assume for $1\leq i\leq \theta,$
\begin{eqnarray}
&&q_{ii}\ \mathrm{has\ odd\ order,}\label{2.7}\\
&&\mathrm{the\ order\ of }\ q_{ii}\ \mathrm{is\ prime\ to\ 3,}\ \mathrm{if}\ A\ \mathrm{is\ of\ type}\ G_2, \label{2.8}
\end{eqnarray}
where $q_{ij}=\chi_j(g_i)$ for $1\leq i,j\leq \theta.$ With these assumptions, we have the following:

\begin{lemma}\cite[Lemma 2.3]{AS2}\label{L2.5}
There exists a root of unit $q$ of odd order and integers $d_i\in \{1,2,3\}, 1\leq i\leq \theta,$ such that for $1\leq i,j\leq \theta,$
$$q_{ii} =q^{2d_i},\  d_ia_{ij}=d_ja_{ji}.$$ Moreover, if $A$ is of type $G_2$.  Then the order of $q$ is prime to $3.$
\end{lemma}
An immediate consequence of Lemma \ref{L2.5} is that the elements $q_{ii},\ 1\leq i\leq \theta,$ have the same order, hence we define
\begin{equation}\label{2.9}
N=|q_{ii}|,\  1\leq i\leq \theta.
\end{equation}

Let $V=V(\mathfrak{D})$ and $\{X_1,\cdots,X_\theta\}$ a canonical basis of $V$. Then the tensor algebra $T(V)$ is a braided Hopf algebra in $_G^G\mathcal{YD}$ with comultiplication determined by
$$\bigtriangleup(X_i)=X_i\otimes 1+1\otimes X_i,\ \ 1\leq i\leq \theta.$$
Since $(ad_cX_i)^{1-a_{ij}}(X_j),\ 1\leq i\neq j\leq \theta,$ are primitive elements in $T(V)$, they generate a braided Hopf idea of $T(V)$, denoted $I$. So we have a quotient braided Hopf algebra:
\begin{equation*}
\mathcal{R}(\mathfrak{D})=T(V)/I
\end{equation*} in $_G^G\mathcal{YD}.$
For convenience,  we still  denote by $X_i$,  $ 1\leq i\leq \theta$,  the image of the element $X_i$ in $\mathcal{R}(\mathfrak{D})$.

Now let $w_0=s_{i_1}s_{i_2}\cdots s_{i_P}$ be a fixed reduced presentation of the longest element of $W$ in terms of simple reflections. Then
\begin{equation}
\beta_l=s_{i_1}\cdots s_{i_{l-1}}(\alpha_{i_l})|1\leq l\leq P
\end{equation}
 is a convex order of positive roots. The root vectors $\{X_\alpha|\alpha\in R^+\}$ can be defined as  iterated braided commutators of the elements $X_1,\cdots ,X_\theta$ with respect to the braiding given by $(q_{ij})_{1\leq i,j\leq n}$ such that $X_{\alpha_i}=X_i, 1\leq i\leq \theta,$ see \cite{AS02,AS2,L1} for detailed definition. Denote by $\mathcal{K}(\mathfrak{D})$ the subalgebra of $\mathcal{R}(\mathfrak{D})$ generated by the elements $Y_l=X_{\beta_l}^N, 1\leq l\leq P.$ The following description of $\mathcal{K}(\mathfrak{D})$ comes from  \cite{AS2}.
\begin{theorem}\cite[Theorem 2.6]{AS2}
\begin{itemize}
\item[(1)] The elements $$X_{\beta_1}^{a_1}X_{\beta_2}^{a_2}\cdots X_{\beta_P}^{a_P},\ a_1,\cdots a_P\geq 0,$$ form a basis of $\mathcal{R}(\mathfrak{D}).$
\item[(2)] $\mathcal{K}(\mathfrak{D})$ is a braided Hopf subalgebra of $\mathcal{R}(\mathfrak{D}).$
\item[(2)] For all $\alpha,\beta\in R^+,$ $[X_\alpha, X_{\beta}^N]_C=0,$ that is,  $X_\alpha X_{\beta}^N-\chi_{\beta}^N(g_\alpha)X_{\beta}^NX_\alpha=0.$
\end{itemize}
\end{theorem}
Let $e_l=(\delta_{kl})_{1\leq k\leq P}\in \mathbb{N}^P,$ where $\delta_{kl}$ is the Kronecker sign. For each $a=(a_1,a_2,\cdots,a_P)\in \mathbb{N}^P,$ define
\begin{eqnarray}
&&Y^a=Y_1^{a_1}Y_2^{a_2}\cdots Y_p^{a_P},\\
&&h^a=h_{\beta_1}^{Na_1}h_{\beta_2}^{Na_2}\cdot h_{\beta_P}^{Na_P},\label{2.14}\\
&&\underline{a}=a_1\beta_1+a_2\beta_2+\cdots +a_P\beta_P.
\end{eqnarray}
Let $\bigtriangleup_{\mathcal{R}(\mathfrak{D})}$ be the comultiplication of $\mathcal{R}(\mathfrak{D}),$ then we have the following lemma.
\begin{lemma}\cite[Lemma 2.8]{AS2}
For any nonzero $a\in \mathbb{N}^P,$ there are uniquely determined scalars $t^a_{b,c}\in \k, 0\neq b,c\in \mathbb{N}^P,$ such that
\begin{equation}
\bigtriangleup_{\mathcal{R}(\mathfrak{D})}(Y^a)=Y^a\otimes 1+1\otimes Y^a+\sum_{b,c\neq 0,\underline{b}+\underline{c}=\underline{a}}t^a_{b,c}Y^b\otimes Y^c.
\end{equation}
\end{lemma}
\begin{definition}
Let $(\mu_a)_{a\in \mathbb{N}^P}$ be a family of elements in $\k$ such that for all $a,$ $h^a=1$ implies $\mu_a=0.$ Then we can define $u^a\in \k G$ inductively on $ht(\underline{a})$ by
\begin{equation}\label{2.17}
u^a=\mu_a(1-h^a)+\sum_{b,c\neq 0,\underline{b}+\underline{c}=\underline{a}}t^a_{b,c}\mu_bu^c.
\end{equation}
\end{definition}

\begin{proposition}\label{P2.9}Let $(\mu_l)_{1\leq l\leq P}$ be a family of elements in $\k$ such that: $g_{\beta_l}^N=1$ or $\chi_{\beta_l}^N\neq \varepsilon$ implies $\mu_l=0.$ Then there exists a unique family $(\mu_a)_{a\in \mathbb{N}^P}$ satisfying $\mu_{e_l}=\mu_l$ for $1\leq l\leq P$ such that
\begin{equation}
u^a=u^{a-e_l}u^{e_l},\ \mathrm{if}\ a=(a_1,\cdots,a_l,0,\cdots,0),\ a_l\geq 1,\ 1\leq l\leq P,\ \mathrm{and}\ a\neq e_l.
\end{equation}
\end{proposition}
\begin{proof}
It follows from \cite[Lemma 2.10, Theorem 2.13]{AS2}
\end{proof}
\begin{definition}\label{D2.10}
Suppose that $\mu=(\mu_l)_{1\leq l\leq P}$ is a family of elements in $\k$ satisfying the condition of Proposition \ref{P2.9}. For a root $\alpha\in R^+,$ then there exists $1\leq l\leq P$ such that $\alpha=\beta_l$ and define $$\mathbf{u}_\alpha(\mu)=u^{e_l}.$$
\end{definition}

\subsection{Andruskiewitsch-Schneider's Hopf algebras}
Fix a datum of finite Cartan type $\mathfrak{D}=\mathfrak{D}(G,(h_i)_{1\leq i\leq \theta},(\chi_i)_{1\leq i\leq \theta},A),$ where $A$ may not be a connected Cartan matrix. For $1\leq i,j\leq \theta$ we define $i\sim j$ if $i$ and $j$ are in the same connected component of the Dynkin diagram of Cartan matrix $A,$ and $i,j$ are said to be connected if $i\sim j.$ Let $\Omega=\{I_1,\cdots,I_t\}$ be the set of the connected components of $I=\{1,2,\cdots,\theta\}$.  Here we also assume that the conditions \eqref{2.7}-\eqref{2.8} hold for each connected component of $I.$ For $J\in \Omega,$ let$R_J$  be the root system of $A_J=(a_{ij})_{i,j\in J}$ and $N_J$ the corresponding number defined by \eqref{2.9}. Let $R_J^+$ be the set of positive roots of $A_J$ with respect to the simple roots $\{\alpha_i|i\in J\}.$   The following partitions are obvious:
$$R=\bigcup_{J\in\Omega} R_J,\  R^+=\bigcup_{J\in \Omega}R^+_J.$$

\begin{definition}
A family $\lambda=(\lambda_{ij})_{1\leq i,j\leq n,i\nsim j}$ of elements in $\k$ is called a family of linking parameters for $\mathfrak{D}$ if $h_ih_j=1$ or $\chi_i\chi_j\neq \varepsilon$ implies $\lambda_{ij}=0$ for all $1\leq i,j\leq \theta, i\nsim j.$ Vertices $1\leq i,j\leq \theta$ are called linkable if $i\nsim j, h_ih_j\neq 1$ and $\chi_i\chi_j=\varepsilon.$
\end{definition}

\begin{definition}
A family $\mu=(\mu_\alpha)_{\alpha\in R^+}$ of elements in $\k$ is called a family of root vector parameters for $\mathfrak{D}$ if $h_{\alpha}^{N_J}=1$ or $\chi_{\alpha}^{N_J}\neq \varepsilon$ implies $\mu_\alpha=0$ for all $\alpha\in R_J^+, J\in \Omega.$
\end{definition}
With these definitions and notations, we can give one of the main results of \cite{AS2}.
\begin{theorem}\cite[Theorem 4.5]{AS2}\label{T2.13}
Let $\mathfrak{D}=\mathfrak{D}(G,(h_i)_{1\leq i\leq \theta},(\chi_i)_{1\leq i\leq \theta},A)$ be a datum of finite Cartan type such that each connected component of $I=\{1,2,\cdots,\theta\}$ satisfies the conditions $\eqref{2.7}$-$\eqref{2.8}.$ Let $\lambda$ and $\mu$ be a family of linking parameters and a family of root vector parameters for $\mathfrak{D}$ respectively. Then we have a finite-dimensional pointed Hopf algebra $u(\mathfrak{D},\lambda,\mu)$ generated by the group $G$ and the skew-primitive elements $\{X_i|1\leq i\leq \theta\}$ subject to the following relations:
\begin{eqnarray*}
&(\mathrm{Action\ of\ the\ group})&gX_ig^{-1}=\chi_i(g)X_i,\ \mathrm{for\ all}\ 1\leq i\leq \theta, g\in G,\\
&(\mathrm{Serre\ relations})&ad_c(X_i)^{1-a_{ij}}(X_j)=0,\ \mathrm{for\ all}\ i\neq j, i\sim j,\\
&(\mathrm{Linking\ relations})&ad_c(X_i)(X_j)=\lambda_{ij}(1-h_ih_j),\ \mathrm{for\ all}\ i<j, i\nsim j,\\
&(\mathrm{Root\ vector\ relations})&X_\alpha^{N_J}=\mathbf{u}_{\alpha}(\mu), \ \mathrm{for\ all}\ \alpha\in R_J^+, J\in \Omega.
\end{eqnarray*}
The coalgebra structure is determined by
$$\bigtriangleup(X_i)=X_i\otimes 1+h_i\otimes X_i,\ \bigtriangleup(g)=g\otimes g, \ \mathrm{for\ all}\ 1\leq i\leq \theta, g\in G.$$
\end{theorem}

The Hopf algebras constructed in  Theorem \ref{T2.13} can be viewed as an axiomatic description of generalized the small quantum groups, and Lusztig's small quantum groups are special examples of such Hopf algebras. Another main result of \cite{AS2} says that  any finite-dimensional pointed Hopf algebra over an abelian group $G$ with dimension not divided by $2,3,5,7,$  is of the form $u(\mathfrak{D},\lambda,\mu)$ for some $\mathfrak{D},\lambda, \mu.$ In the sequel, we call such a  Hopf algebra $u(\mathfrak{D},\lambda,\mu)$ an Andruskiewitsch-Schneider Hopf algebra (or AS-Hopf algebra for short) following \cite{EG3}.

\subsection{Normalized $3$-cocycles on finite groups}

Let $G$ be an arbitrary abelian group. So $G\cong \mathbb{Z}_{m_{1}}\times\cdots \times\mathbb{Z}_{m_{n}}$ with $m_j\in \mathbb{N}$
for $1\leq j\leq n.$ A function $\phi:G\times G\times G\mapsto \k^*$ is called a 3-cocycle on $G$ if
\begin{equation}
\phi(ef,g,h)\phi(e,f,gh)=\phi(e,f,g)\phi(e,fg,h)\phi(f,g,h)
\end{equation}
for all $e,f,g,h \in G$. A 3-cocycle  is called normalized if $\phi(f,1,g)=1.$  Denote by $\mathcal{A}$ the set of all sequences
\begin{equation}\label{cs}(c_{1},\ldots,c_{l},\ldots,c_{n},c_{12},\ldots,c_{ij},\ldots,c_{n-1,n},c_{123},
\ldots,c_{rst},\ldots,c_{n-2,n-1,n})\end{equation}
such that $ 0\leq c_{l}<m_{l}, \ 0\leq c_{ij}<(m_{i},m_{j}), \ 0\leq c_{rst}<(m_{r},m_{s},m_{t})$ for $1\leq l\leq n, \ 1\leq i<j\leq n, \ 1\leq r<s<t\leq n$, where $c_{ij}$ and $c_{rst}$ are ordered in the lexicographic order. We denote by $\underline{\mathbf{c}}$ the sequence \eqref{cs}  in the following.

Let $g_i$ be a generator of $\mathbb{Z}_{m_{i}}, 1\leq i\leq n$. For any $\underline{\mathbf{c}}\in A$, define
\begin{eqnarray}
&& \omega_{\underline{\mathbf{c}}}:\;G\times G\times G\To \k^{\ast} \notag \\
&&[g_{1}^{i_{1}}\cdots g_{n}^{i_{n}},g_{1}^{j_{1}}\cdots g_{n}^{j_{n}},g_{1}^{k_{1}}\cdots g_{n}^{k_{n}}] \mapsto \\ &&\prod_{l=1}^{n}\zeta_{m_l}^{c_{l}i_{l}[\frac{j_{l}+k_{l}}{m_{l}}]}
\prod_{1\leq s<t\leq n}\zeta_{m_{t}}^{c_{st}i_{t}[\frac{j_{s}+k_{s}}{m_{s}}]}
\prod_{1\leq r<s<t\leq n}\zeta_{(m_{r},m_{s},m_{t})}^{c_{rst}k_{r}j_{s}i_{t}}. \notag
\end{eqnarray}
Here and below $\zeta_m$ stands for an $m$-th primitive root of unity.

\begin{proposition}\cite[Proposition 3.1]{bgrc2} \label{P2.14}  $\{\omega_{\underline{\mathbf{c}}}|\underline{\mathbf{c}}\in A\}$ forms a complete set of representatives of the normalized $3$-cocycles on $G$ up to $3$-cohomology.
\end{proposition}
For the purpose of this paper, we need another class of representatives of the normalized $3$-cocycles on $G.$
Let $Z^3(G,\k^*)$ be the set of the normalized $3$-cocycles on $G$. Define a map
\begin{equation}
\sigma: Z^3(G,\k^*)\To Z^3(G,\k^*),\ \
\sigma(\phi)(f,g,h)=\phi(h,g,f),\ \mathrm{for\ all}\ f,g,h\in G.
\end{equation}
To see if the map is well-defined, we just need to show that $\sigma(\phi)$ is a normalized $3$-cocycle on $G$ for each $\phi\in Z^3(G,\k^*).$
Indeed,  for any $e,f,g,h\in G$ and $\phi\in Z^3(G,\k^*),$ we have
\begin{eqnarray*}
\partial(\sigma(\phi))(e,f,g,h)&=&\frac{\sigma(\phi)(e,f,g)\sigma(\phi)(e,fg,h)\sigma(\phi)(f,g,h)}{\sigma(\phi)(ef,g,h)\sigma(\phi)(e,f,gh)}\\
&=&\frac{\phi(g,f,e)\phi(h,fg,e)\phi(h,g,f)}{\phi( h,g,ef)\phi(gh,f,e)}\\
&=&\partial(\phi)(h,g,f,e)\\
&=&1.
\end{eqnarray*}
This implies that $\sigma(\phi)$ is a $3$-cocycle on $G$. The fact that  $\sigma(\phi)$ is normalized  follows from  the equation:
$$\sigma(\phi)(f,1,g)=\phi(g,1,f)=1,\ \mathrm{for\ all}\ f,g\in G.$$
It is obvious that $\sigma$ is bijective since $\sigma^2=id.$ Moreover,  we have the following.
\begin{lemma}\label{L2.15}
The map $\sigma$ induces an involution of $H^3(G,\k^*).$
\end{lemma}
\begin{proof}
It suffices show that $\sigma$ preserves $3$-coboundaries. Suppose that $\phi$ is a $3$-coboundary.  There exists a $2$-cochain $J:G\times G\to \k^*$ such that $\phi=\partial(J).$ Define $J':G\times G\to \k^*$ by
\begin{equation}
J'(f,g)=J^{-1}(g,f),\ \mathrm{for\ all}\ f,g\in G.
\end{equation}
Then we have:
\begin{eqnarray*}
\sigma(\phi)(f,g,h)=\frac{J(h,g)J(gh,f)}{J(g,f)J(h,fg)}=\frac{J'^{-1}(g,h)J'^{-1}(f,gh)}{J'^{-1}(f,g)J'^{-1}(fg,h)}
=\frac{J'(f,g)J'(fg,h)}{J'(g,h)J'(f,gh)}=\partial(J')(f,g,h)
\end{eqnarray*} for all $f,g,h\in G.$ This implies that $\sigma$ preserves $3$-coboundaries. Thus,  we have completed the proof.
\end{proof}
For each $\underline{\mathbf{c}}\in \A$, define
\begin{eqnarray}
&& \phi_{\underline{\mathbf{c}}}:\;G\times G\times G\To \k^{\ast}, \notag \\
&&[g_{1}^{i_{1}}\cdots g_{n}^{i_{n}},g_{1}^{j_{1}}\cdots g_{n}^{j_{n}},g_{1}^{k_{1}}\cdots g_{n}^{k_{n}}] \mapsto \\ &&\prod_{l=1}^{n}\zeta_{m_l}^{c_{l}k_{l}[\frac{i_{l}+j_{l}}{m_{l}}]}
\prod_{1\leq s<t\leq n}\zeta_{m_{t}}^{c_{st}k_{t}[\frac{i_{s}+j_{s}}{m_{s}}]}
\prod_{1\leq r<s<t\leq n}\zeta_{(m_{r},m_{s},m_{t})}^{c_{rst}i_{r}j_{s}k_{t}}. \notag
\end{eqnarray}
It is obvious that $\sigma(\omega_{\underline{\mathbf{c}}})=\phi_{\underline{\mathbf{c}}}$ for each $\underline{\mathbf{c}}\in \A.$  It follows from Proposition \ref{P2.14} and Lemma \ref{L2.15}  that we have the following:

\begin{proposition}\label{P2.16}  $\{\phi_{\underline{\mathbf{c}}}|\underline{\mathbf{c}}\in \A\}$ forms a complete set of representatives of the normalized $3$-cocycles on $G$ up to $3$-cohomology.
\end{proposition}

The original definition of an {\it abelian cocycle} was given in \cite{EM}, and an equivalent definition via  the twisted quantum double appeared in \cite{MN}. Let $\phi$ be a $3$-cocycle on $G$, and $D^\phi(G)$ the twisted quantum double of $(\k G, \phi)$ (see \cite{HLYY} for the detail).  $\phi$ is called an {\it  abelian $3$-cocycle} if $D^\phi(G)$ is commutative. Using Proposition \ref{P2.16}, one can easily determine all the abelian $3$-cocycles on $G$. A straightforward computation  shows that  $\phi_{\underline{\mathbf{c}}}$ is an abelian $3$-cocycle if and only if $c_{rst}=0$
for all $1\leq r<s<t\leq n$. We point out that  the twisted Yetter-Drinfeld category $_G^G\mathcal{YD}^{\phi_{\underline{\mathbf{c}}}}$ is a pointed fusion category in case the $3$-cocycle $\phi_{\underline{\mathbf{c}}}$ is abelian.

Denote by $\Vec_G$ the category of $G$-graded vector spaces. Let $\omega$ be a $3$-cocycle on $G$. We define a tensor category $\Vec_G^\omega$. As a category,  $\Vec_G^\omega=\Vec_G$.  The tensor product $V\otimes W$ of two graded modules is endowed with the canonical grading:
 $$(V\otimes W)_g=\oplus_{ef=g}V_e\otimes V_f, \ \forall g\in G.$$  The associator $\mathfrak{a}$ is given by
\begin{eqnarray*}
\mathfrak{a}_{U,V,W}: (U\otimes V)\otimes W&\To& U\otimes (V\otimes W)\\
(x\otimes y)\otimes z &\mapsto& \omega^{-1}(e,f,g) x\otimes (y\otimes z),
\end{eqnarray*} where $x\in U_e,y\in V_f,z\in W_g.$
According to \cite[Proposition 2.6.1]{EGNO}, $\Vec_G^\omega$ is tensor equivalent to the representation category of some Hopf algebra if and only if $\phi$ is a $3$-coboundary on $G$.

\section{Finite-dimensional Quasi-Hopf algebras}

\subsection{General setup} In this subsection, we fix some notations on abelian groups, which will be used throughout  this paper.
Suppose  that $G$ is a finite abelian group, say, $G=\langle g_1\rangle \times \cdots\times \langle g_n\rangle$ such that $|g_i|=m_i$ for $1\leq i\leq n.$  Let $\widehat{G}$ be the character group of $G$ over $\k.$ For each $g=\prod_i^ng_i^{\alpha_i},$ define a character $\chi_g:G\to \k^*$ by
  \begin{equation}\label{3.1}
  \chi_g(h)=\prod_i^n\zeta_{m_i}^{\alpha_i\beta_i},
  \end{equation} where $h=\prod_i^ng_i^{\beta_i}\in G.$ From the definition of $\chi_g$, it is obvious that $\chi^{-1}_g(h)=\chi_{g^{-1}}(h)=\chi_g(h^{-1}).$
  So $\chi:G\To \widehat{G},g\to \chi_g$ is an group isomorphism. Let $\k[G]$ be the group algebra of $G$ over field $\k.$  One can verify that
  \begin{equation}\label{3.2}
  \{1_g=\frac{1}{|G|}\sum_{h\in G}\chi_g(h)h|g\in G\}
  \end{equation} forms a complete set of the orthogonal primitive idempotents of the algebra $\k[G].$
\begin{lemma}\label{L3.1}
For $g,h\in G,$ we have $1_gh=h1_g=\chi^{-1}_g(h)1_g.$
\end{lemma}
\begin{proof}
$1_gh=\frac{1}{|G|}\sum_{f\in G}\chi_g(f)fh=\frac{1}{|G|}\sum_{f\in G}\chi_g(fh)\chi_g(h^{-1})fh=\chi^{-1}_g(h)1_g.$
\end{proof}

Now let $G$ be an abelian group. We can define a bigger abelian group $\mathbb{G}$ associated to $G$ in the following way:  assume
\begin{equation}
G=\langle g_1\rangle \times \cdots\times \langle g_n\rangle,\ |g_i|=m_i, 1\leq i\leq n;
\end{equation}  define the group $\mathbb{G}$  as follows:
\begin{equation}\label{e3.4}
 \mathbbm{G}=\langle \mathbbm{g}_1\rangle\times \cdots\times \langle \mathbbm{g}_n\rangle,\ |\mathbbm{g}_i|=\mathbbm{m}_i=m_i^2, 1\leq i\leq n.
\end{equation}
It is obvious that there is a group injection:
\begin{equation}\label{3.3}
\iota:G\to \mathbbm{G},\ \iota(g_i)=\mathbbm{g}_i^{m_i},\ 1\leq i\leq n.
\end{equation}
Let $\zeta_{\mathbbm{m}_i}$ be an $\mathbbm{m}_i$-th primitive root of unity such that $\zeta_{\mathbbm{m}_i}^{m_i}=\zeta_{m_i}$ for $1\leq i\leq n.$ For each $\mathbbm{g}=\prod_{i=1}^n\mathbbm{g}_i^{s_i}\in \mathbbm{G},$ define $\chi_{\mathbbm{g}}:\mathbbm{G}\to \k^*$ by
$$\chi_{\mathbbm{g}}(h)=\prod_{i=1}^{n}\zeta_{\mathbbm{m}_i}^{s_it_i},\ h=\prod_{i=1}^n\mathbbm{g}_i^{t_i}.$$
Similar to \eqref{3.2},  one has  a complete set $\{\mathbbm{1}_{\mathbbm{g}}=\frac{1}{|\mathbbm{G}|}\sum_{h\in \mathbbm{G}}\chi_{\mathbbm{g}}(h)h|\{\mathbbm{1}_{\mathbbm{g}}|\mathbbm{g}\in \mathbbm{G}\}\in \mathbbm{G}\}$ of orthogonal primitive idempotents of the algebra $\k[\mathbbm{G}].$
We have the following equality.
\begin{lemma}\label{L3.2}
 The following holds for all $0\leq s_i\leq m_i-1,\ 1\leq i\leq n$:
\begin{equation}
\sum_{0\leq k_j \leq m_j-1,\\ 1\leq j\leq n}\mathbbm{1}_{(\prod_{i=1}^n\mathbbm{g}_i^{m_ik_i+s_i})}=1_{(\prod_{i=1}^ng_i^{s_i}).}
\end{equation}
\end{lemma}
\begin{proof}By definition we have
$$\sum_{0\leq k_j \leq m_j-1,\\ 1\leq j\leq n}\mathbbm{1}_{(\prod_{i=1}^n\mathbbm{g}_i^{m_ik_i+s_i})}=\frac{1}{|\mathbbm{G}|}\sum_{0\leq k_j \leq m_j-1,\\ 1\leq j\leq n}\sum_{h\in \mathbbm{G}}\chi_{(\prod_{i=1}^n\mathbbm{g}_i^{m_ik_i+s_i})}(h)h.$$
Suppose $h=\mathbbm{g}_1^{r_1}\mathbbm{g}_2^{r_2}\cdots \mathbbm{g}_n^{r_n}.$ Then we have the equation:
\begin{eqnarray*}
&&\sum_{0\leq k_j \leq m_j-1, 1\leq j\leq n}\chi_{(\prod_{i=1}^n\mathbbm{g}_i^{m_ik_i+s_i})}(h)\\
&&=\sum_{0\leq k_j \leq m_j-1, 1\leq j\leq n}\prod_i^n\zeta_{\mathbbm{m}_i}^{(m_ik_i+s_i)r_i}\\
&&=\prod_{i=1}^n(\sum_{0\leq k_i \leq m_i-1}\zeta_{\mathbbm{m}_i}^{(m_ik_i+s_i)r_i}).
\end{eqnarray*}
Note that  $\sum_{0\leq k_l \leq m_l-1}\zeta_{\mathbbm{m}_l}^{m_lk_lr_l}=0$ if $r_l\neq tm_l$ for some integer $0\leq t\leq m_l-1$.  Hence $$\sum_{0\leq k_j \leq m_j-1, 1\leq j\leq n}\chi_{(\prod_{i=1}^n\mathbbm{g}_i^{m_ik_i+s_i})}(h)\neq 0$$ if and only if $r_i=t_im_i$ for $0\leq t_i\leq m_i-1, 1\leq i\leq n,$ i.e., $h$ is contained in the subgroup $G.$  If $r_i=t_im_i$ for $1\leq i\leq n,$ then $h=\mathbbm{g}_1^{t_1m_1}\mathbbm{g}_2^{t_2m_2}\cdots \mathbbm{g}_n^{t_nm_n}$ and we have:
\begin{eqnarray*}
\sum_{0\leq k_j \leq m_j-1, 1\leq j\leq n}\chi_{(\prod_{i=1}^n\mathbbm{g}_i^{m_ik_i+s_i})}(h)&=&\prod_{i=1}^n(\sum_{0\leq k_i \leq m_i-1}\zeta_{\mathbbm{m}_i}^{(m_ik_i+s_i)t_im_i})\\
&=&\prod_{i=1}^n(\sum_{0\leq k_i \leq m_i-1}\zeta_{\mathbbm{m}_i}^{s_it_im_i})\\
&=&\prod_{i=1}^n(m_i\zeta_{m_i}^{s_it_i})\\
&=&|G|\prod_{i=1}^n\zeta_{m_i}^{s_it_i}.
\end{eqnarray*}
It follows that
\begin{eqnarray*}
\sum_{0\leq k_j \leq m_j-1, 1\leq j\leq n}\mathbbm{1}_{(\prod_{i=1}^n\mathbbm{g}_i^{m_ik_i+s_i})}
&=&\frac{1}{|\mathbbm{G}|}\sum_{0\leq k_j \leq m_j-1, 1\leq j\leq n}\sum_{h\in G}\chi_{(\prod_{i=1}^n\mathbbm{g}_i^{m_ik_i+s_i})}(h)h.\\
&=&\frac{|G|}{|\mathbbm{G}|}\sum_{0\leq t_j\leq m_j-1,1\leq j\leq n}\prod_{i=1}^n\zeta_{m_i}^{s_it_i}(\prod_{i=1}^n\mathbbm{g}_i^{t_im_i})\\
&=&\frac{1}{|G|}\sum_{0\leq t_j\leq m_j-1,1\leq j\leq n}\chi_{(\prod_{i=1}^ng_i^{s_i})}(\prod_{i=1}^ng_i^{t_i})(\prod_{i=1}^ng_i^{t_i})\\
\end{eqnarray*}

\begin{eqnarray*}
&=&\frac{1}{|G|}\sum_{g\in G}\chi_{(\prod_{i=1}^ng_i^{s_i})}(g)g\\
&=&1_{(\prod_{i=1}^ng_i^{s_i})}.
\end{eqnarray*}
Thus,  the claimed equality holds.
\end{proof}

\subsection{Finite dimensional quasi-Hopf algebras and Cartan matrices}
Keep the notations of the last subsection. Let  $\mathfrak{D}=\mathfrak{D}(\mathbbm{G},(h_i)_{1\leq i\leq \theta},(\chi_i)_{1\leq i\leq \theta},A)$ be a datum of finite Cartan type, where $A=(a_{ij})_{1\leq i,j\leq \theta}$ is a finite Cartan matrix. Let $\Omega$ be the set of the connected components of $I=\{1,2,\cdots,\theta\}.$ Denote by $(s_{ij})_{1\leq i\leq \theta,1\leq j\leq n}$ and $(r_{ij})_{1\leq i\leq \theta,1\leq j\leq n}$ the two families of integers satisfying:
\begin{eqnarray}
&&h_i=\prod_{j=1}^n\mathbbm{g}_j^{s_{ij}},\ \chi_i(\mathbbm{g}_j)=\zeta_{\mathbbm{m}_j}^{r_{ij}},\label{3.5}\\
&&0\leq s_{ij}, r_{ij}<\mathbbm{m}_j\ \mathrm{ for\ all}\ 1\leq i\leq \theta, 1\leq j\leq n.
\end{eqnarray}
It is obvious that $(s_{ij})_{1\leq i\leq \theta,1\leq j\leq n}$ and $(r_{ij})_{1\leq i\leq \theta,1\leq j\leq n}$ are uniquely determined by $\mathfrak{D}$.
Let $\Gamma(\mathfrak{D})$ be the subset of $\mathcal{A}$ such that for each $\underline{c}\in \Gamma(\mathfrak{D}),$ $ c_{rst}=0$ for all $ 0\leq r<s<t\leq n$ and
\begin{eqnarray}
&& s_{ij}\equiv c_jr_{ij} \mod m_j, 1\leq i\leq \theta, 1\leq j\leq n,\label{3.7}\\
&& c_{ij}r_{lj}\equiv 0 \mod m_j, 1\leq l\leq \theta, 1\leq i< j\leq n,\label{3.8}\\
&& c_{ij}m_{i}\equiv 0 \mod m_j,  1\leq i< j\leq n.\label{3.9}
\end{eqnarray}
For each $\underline{c}\in \Gamma(\mathfrak{D}),$ define  on $G$ the functions $\Theta, \Psi_l,  \Upsilon$ and $\mathcal{F}_i$  as follows:
\begin{eqnarray}
&&\Theta_l(g)=\prod_{i=1}^{n}\zeta_{\mathbbm{m}_i}^{-t_is_{li}}, \  g=\prod_{i=1}^ng_i^{t_i}\in G,1\leq l\leq \theta.\\
&&\Psi_l(f,h)=\prod_{i=1}^{n}\zeta_{\mathbbm{m}_i}^{c_iq_i\varphi_{li}(f)}\prod_{1\leq i<j\leq n}\zeta_{m_im_j}^{c_{ij}q_j\varphi_{li}(f)}, \
f=\prod_{i=1}^ng_i^{p_i},\ h=\prod_{i=1}^ng_i^{q_i}\in G,1\leq l\leq \theta.\label{e3.14}\\
&&\Upsilon(g)=\prod_{i=1}^n\zeta_{\mathbbm{m}_i}^{-c_il_im_i}\prod_{1\leq s<t\leq n}\zeta_{m_sm_t}^{-c_{st}l_tm_s}, \ g=\prod_{i=1}^ng_i^{l_i}.\label{e3.15}\\
&&\mathcal{F}_i(g)=\prod_{j=1}^n\zeta_{\mathbbm{m}_j}^{-c_j(k_j-r_{ij})\varphi_{ij}(g)}
\prod_{1\leq s<t\leq n}\zeta_{\mathbbm{m}_s\mathbbm{m}_t}^{-c_{st}(k_t-r_{it})\varphi_{is}(g)}, \ g=\prod_{i=l}^ng_i^{k_l}.\label{e3.16}
\end{eqnarray}
Here \begin{equation}
\varphi_{li}(g)=(k_i-r_{li})'-(k_i-r_{li}), \ g=\prod_{i=l}^ng_i^{k_l},
\end{equation}
and $(p_i-r_{li})'$ is the remainder of $p_i-r_{li}$ divided by $m_i.$

In order to construct quasi-Hopf algebras, we need the notions of modified linking parameters and modified root vector parameters.
\begin{definition}
A family of linking parameters $\lambda=(\lambda_{ij})_{1\leq i,j\leq \theta}$ for $\mathfrak{D}$ is said to be modified if
\begin{equation}\label{3.12}
h_ih_j\notin G\ \mathrm{implies}\ \lambda_{ij}=0\ \mathrm{for\ all}\ 1\leq i<j\leq \theta, i\nsim j.
\end{equation}
A family of root vector parameters $\mu=(\lambda_{\alpha})_{\alpha\in R}$ for $\mathfrak{D}$ is said to be modified if
\begin{equation}\label{3.13}
h_\alpha^{N_J}\notin G\ \mathrm{implies}\ \mu_{\alpha}=0\ \mathrm{for\ all}\ \alpha\in R_J^+, J\in \Omega.
\end{equation}
\end{definition}
Now we can give the main result of this paper.

\begin{theorem}\label{T3.4}
 Let $\lambda=(\lambda_{ij})_{1\leq i,j\leq \theta}$ and $\mu=(\lambda_{\alpha})_{\alpha\in R}$ be two families of  modified linking parameters and root vector parameters respectively  for a datum of Cartan type $\mathfrak{D}=\mathfrak{D}(\mathbbm{G},(h_i)_{1\leq i\leq \theta},(\chi_i)_{1\leq i\leq \theta},A)$, and $\underline{c}$ a nonzero element in $\Gamma(\mathfrak{D}).$ Then we have a finite-dimensional quasi-Hopf algebra $u(\mathfrak{D},\lambda,\mu,\Phi_{\underline{c}})$ generated by $G$ and $\{X_1,\cdots,X_\theta\}$ subject to the relations:
\begin{eqnarray}
&&gX_ig^{-1}=\chi_i(g)X_i,\ \mathrm{for\ all}\ 1\leq i\leq \theta, g\in G,\label{e3.20}\\
&&ad_C(X_i)^{1-a_{ij}}(X_j)=0,\ \mathrm{for\ all}\ i\neq j, i\sim j,\label{e3.21}\\
&&ad_C(X_i)(X_j)=\lambda_{ij}(1-h_ih_j),\ \mathrm{for\ all}\ i<j, i\nsim j,\label{e3.22}\\
&&X_\alpha^{N_J}=\mathbf{u}_{\alpha}(\mu), \ \mathrm{for\ all}\ \alpha\in R_J^+, J\in \Omega.\label{e3.23}
\end{eqnarray}
The coalgebra structure of $u(\mathfrak{D},\lambda,\mu,\Phi_{\underline{c}})$ is given  by
\begin{equation}\label{3.18}\bigtriangleup(X_l)=\sum_{f,g\in G}\Psi_l(f,g)X_l1_f\otimes 1_g+\sum_{f\in G}\Theta_l(f)1_f\otimes X_l, \bigtriangleup(g)=g\otimes g, 1\leq l\leq \theta, g\in G.
\end{equation}
The associator  of $u(\mathfrak{D},\lambda,\mu,\Phi_{\underline{c}})$ is determined by
\begin{equation}\label{3.19}
\Phi_{\underline{c}}=\sum_{f,g,h\in G}\phi_{\underline{c}}(f,g,h)(1_f\otimes 1_g\otimes 1_h).
\end{equation}
The antipode $(\mathcal{S},\mathcal{\alpha},1)$ is  defined by
\begin{equation}\label{3.20}
\mathcal{\alpha}=\sum_{g\in G}\Upsilon(g)1_g,\ \mathcal{S}(X_i)=\sum_{g\in G}\mathcal{F}_i(g)X_i1_g.
\end{equation}
\end{theorem}

The proof of Theorem \ref{T3.4}  will be delivered in the next subsection.  Since the quasi-Hopf algebra $u(\mathfrak{D},\lambda,\mu,\Phi_{\underline{c}})$ is generated by the abelian group $G$ and the braided vector space of Cartan type $V=\k\{X_1,\cdots,X_\theta\},$  we  shall call it a quasi-Hopf algebra of Cartan type in the sequel. We will say that $u(\mathfrak{D},\lambda,\mu,\Phi_{\underline{c}})$ is {\bf{associated to the Cartan matrix}} $A$, and call $\theta$ the {\bf{rank}} of $u(\mathfrak{D},\lambda,\mu,\Phi_{\underline{c}}),$ as well as  the rank of $A$.

\begin{remark}
The relations (3.20)-(3.23) for $u(\mathfrak{D},\lambda,\mu,\Phi_{\underline{c}})$ are similar to those of AS-Hopf algebras $u(\mathfrak{D},\lambda,\mu),$ but the generators of the two algebras are essential different. In fact, $u(\mathfrak{D},\lambda,\mu)$ is generated by $\mathbb{G}$ and $\{X_1,\cdots,X_\theta\},$ and $u(\mathfrak{D},\lambda,\mu,\Phi_{\underline{c}})$ is the subalgebra of $u(\mathfrak{D},\lambda,\mu)$ generated by subgroup $G$ and $\{X_1,\cdots,X_\theta\}.$ Moreover, we will prove that $u(\mathfrak{D},\lambda,\mu,\Phi_{\underline{c}})$ is a quasi-Hopf subalgebra of $u(\mathfrak{D},\lambda,\mu)^J$ for some twist $J$ of $u(\mathfrak{D},\lambda,\mu).$
\end{remark}

It is obvious that $u(\mathfrak{D},\lambda,\mu,\Phi_{\underline{c}})$ is {\bf{radically\ graded}} if and only if
$u_\alpha(\mu)=0, \lambda_{ij}(1-h_ih_j)=0,$ for all $\alpha\in R^+, 1\leq i,j\leq \theta,$ and the radical is the ideal generated by $X_1,\cdots, X_{\theta}.$  Note that if $\mu=0,$ then $u_\alpha(\mu)=0$ by \eqref{2.14}.  It follows  the definitions of a familly  of modified linking parameters and a family of  modified root vector parameters thta the quasi-hopf algebra $u(\mathfrak{D},\lambda,\mu,\Phi_{\underline{c}})$ is radically graded if and only if  both $\lambda=0$ and $\mu=0$.
\begin{remark}
\begin{itemize}
\item[(1).] Let $G$ be a cyclic group, $H$ the AS-Hopf algebra $u(\mathfrak{D},0,0)$.  The quasi-Hopf algebra $u(\mathfrak{D},0,0,\Phi_{\underline{c}})$ is nothing but the basic quasi-Hopf algebra $A(H,c)$ (over the cyclic group $G$) classified in \cite{A}.
\item[(2).] Suppose that $A$ is a connected Cartan Matrix of rank $n$, and $\mathfrak{g}$ is the simple Lie algebra associated to $A$. Let $G=\Z_m^n$ for some positive odd integer $m,$ which is not divisible by $3$ if $A$ is of type $G_2.$ Let $h_i=\prod_{j=1}^n\mathbbm{g}_j^{a_{ij}},$ $\chi_i(\mathbbm{g}_l)=\zeta_{m^2}^{\delta_{il}}$ for $1\leq i,l\leq n.$ Let $c_i=a_{ii}, c_{j,k}=a_{jk}$ for $1\leq i\leq n, 1\leq j<k\leq n.$ Then $u(\mathfrak{D},0,0,\Phi_{\underline{c}})$ is the half small quasi-quantum groups $A_q(\mathfrak{g})$ given in \cite{EG2}, where $q=\zeta_{m^2}$.
\item[(2).] Suppose that $A$ is the diagonal Cartan matrix $A_1\times A_1\times \cdots A_1$.  Then the dual of the quasi-Hopf algebras $u(\mathfrak{D},0,0,\Phi_{\underline{c}})$ is  a quasi-quantum linear space, see \cite{HY}.
\end{itemize}
\end{remark}

\subsection{The proof of Theorem 3.4}
Let $H=u(\mathfrak{D},\lambda,\mu)$ be the AS-Hopf algebra given in Theorem \ref{T2.13}, which is generated by $\mathbbm{G}$ and $\{X_1,\cdots,X_\theta\}.$ By  \cite[Theorem 4.5]{AS2}, the group of group-like elements of $H$ is $\mathbbm{G}$. According to Subsection 3.1, we know that $\chi:\mathbbm{G}\to \widehat{\mathbbm{G}}$ is a group isomorphism. Let $\{\eta_1,\cdots,\eta_\theta\}$ be the set of elements in $\mathbbm{G}$ such that $\chi_{\eta_i}=\chi_i$ for $1\leq i\leq \theta.$ By \eqref{3.5}, it is obvious that
\begin{equation}
\eta_i=\prod_{j=1}^n\mathbbm{g}_j^{r_{ij}}, 1\leq i\leq \theta.
\end{equation} Moreover we have the following.
\begin{lemma}\label{L3.5}
\begin{equation}\label{3.14}
\mathbbm{1}_{\mathbbm{g}}X_i=X_i\mathbbm{1}_{\mathbbm{g}\eta_i},\ \mathrm{for\ all}\ \mathbbm{g}\in \mathbbm{G}, 1\leq i\leq \theta.
\end{equation}
\end{lemma}
\begin{proof}  Follows from the following equations:
\begin{eqnarray*}
\mathbbm{1}_{\mathbbm{g}}X_i=\frac{1}{|\mathbbm{G}|}\sum_{\mathbbm{f}\in \mathbbm{G}}\chi_{\mathbbm{g}}(\mathbbm{f})\mathbbm{f}X_i=\frac{1}{|\mathbbm{G}|}\sum_{\mathbbm{f}\in \mathbbm{G}}\chi_{\mathbbm{g}}(\mathbbm{f})\chi_{\eta_i}(\mathbbm{f})X_i\mathbbm{f}=\frac{1}{|\mathbbm{G}|}X_i\sum_{\mathbbm{f}\in \mathbbm{G}}\chi_{\mathbbm{g}\eta_i}(\mathbbm{f})\mathbbm{f}
=X_i\mathbbm{1}_{\mathbbm{g}\eta_i}.
\end{eqnarray*}
\end{proof}

Given $\underline{c}\in \Gamma(\mathfrak{D}),$   we define:
\begin{equation}\label{3.15} J_{\underline{c}}:\; \mathbbm{G}\times \mathbbm{G}\to k^{\ast};\;\;\;\;(\mathbbm{g}_{1}^{x_{1}}\cdots \mathbbm{g}_{n}^{x_{n}},\mathbbm{g}_{1}^{y_{1}}\cdots \mathbbm{g}_{n}^{y_{n}})\mapsto
\prod_{l=1}^{n}\zeta_{\mathbbm{m}_l}^{c_{l}y_{l}(x_l'-x_l)}
\prod_{1\leq s<t\leq n}\zeta_{m_sm_t}^{c_{st}y_{t}(x_s'-x_s)},
\end{equation}
where the element  $x'_i$  stands for the remainder of $x_i$ divided by $m_i$,  for  $1\leq i\leq n$.  One can easily verify that $J_{\underline{c}}$ is a 2-cochain of $\mathbbm{G}.$
We will see that this 2-cochain $J_{\underline{c}}$ induces a twist of the Hopf algebra $H$. Define $\mathbbm{J}_{\underline{c}}=\sum_{\mathbbm{f,g}\in \mathbbm{G}}J_{\underline{c}}(\mathbbm{f,g})\mathbbm{1}_\mathbbm{f}\otimes \mathbbm{1}_\mathbbm{g}\in H\otimes H$.

\begin{lemma}\label{L3.6}
$\mathbbm{J}_{\underline{c}}$ is a twist of $H.$
\end{lemma}
\begin{proof}
It is obvious that $\mathbbm{J}_{\underline{c}}$ is invertible with inverse $\mathbbm{J}^{-1}_{\underline{c}}=\sum_{\mathbbm{f,g}\in \mathbbm{G}}J^{-1}_{\underline{c}}(\mathbbm{f,g})\mathbbm{1}_\mathbbm{f}\otimes \mathbbm{1}_\mathbbm{g}.$
Next we verify that $(\varepsilon\otimes id)(\mathbbm{J}_{\underline{c}})=(id\otimes \varepsilon)(\mathbbm{J}_{\underline{c}})=1$ holds.
Suppose $\mathbbm{g}=\prod_{i=1}^n\mathbbm{g}_i^{k_i}$ for some $0\leq k_i\leq \mathbbm{m}_i, 1\leq i\leq n.$  Then  the following equations hold:
\begin{eqnarray*}
\varepsilon(\mathbbm{1}_\mathbbm{g})&=&\frac{1}{|\mathbbm{G}|}\varepsilon(\sum_{\mathbbm{h}\in \mathbbm{G}}\chi_{\mathbbm{g}}(\mathbbm{h})\mathbbm{h})=\frac{1}{|\mathbbm{G}|}\sum_{\mathbbm{h}\in \mathbbm{G}}\chi_{\mathbbm{g}}(\mathbbm{h})\\
&=& \frac{1}{|\mathbbm{G}|}\sum_{0\leq l_j\leq \mathbbm{m}_j-1,1\leq j\leq n}(\prod_{i=1}^n\zeta_{\mathbbm{m}_i}^{k_il_i})\\
&=& \frac{1}{|\mathbbm{G}|}\prod_{i=1}^n(\sum_{0\leq l_i\leq \mathbbm{m}_i-1}\zeta_{\mathbbm{m}_i}^{k_il_i}).
\end{eqnarray*}
 It follows that
\begin{equation*}
\varepsilon(\mathbbm{1}_{\mathbbm{h}})=\left\{
                                         \begin{array}{ll}
                                           0, & \hbox{if $\mathbbm{h}\neq 1$;} \\
                                           1, & \hbox{if $\mathbbm{h}=1$.}
                                         \end{array}
                                       \right.
\end{equation*}
Hence,  $(\varepsilon\otimes id)(\mathbbm{J}_{\underline{c}})=\sum_{\mathbbm{g}\in \mathbbm{G}}J_{\underline{c}}(1,\mathbbm{g})\mathbbm{1}_{\mathbbm{g}}=\sum_{\mathbbm{g}\in \mathbbm{G}}\mathbbm{1}_{\mathbbm{g}}=1.$
Similarly,  the equation: $( id\otimes\varepsilon)(\mathbbm{J}_{\underline{c}})=1$ holds.
\end{proof}
Since $H$ is a Hopf algebra, we can view it as a quasi-Hopf algebras with the trivial associator $\Phi=1\otimes 1\otimes 1$ and the usual  antipode $(S, 1,1)$.
According to Subsection 2.1,  we can construct a quasi-Hopf algebra
$H^{\mathbbm{J}_{\underline{c}}}=(H^{\mathbbm{J}_{\underline{c}}},\bigtriangleup_{\mathbbm{J}_{\underline{c}}},\varepsilon,\Phi_{\mathbbm{J}_{\underline{c}}},
S_{\mathbbm{J}_{\underline{c}}},\beta_{\mathbbm{J}_{\underline{c}}}\alpha_{\mathbbm{J}_{\underline{c}}},1).$
The associator of $H^{\mathbbm{J}_{\underline{c}}}$  can be explicitly described as follows:

\begin{lemma}\label{L3.7}
\begin{equation}\label{3.22}
\Phi_{\mathbbm{J}_{\underline{c}}}=\sum_{f,g,h\in G}\phi_{\underline{c}}(f,g,h)1_f\otimes 1_g\otimes 1_h
\end{equation}
\end{lemma}
\begin{proof}
First of all, we need to verify  the comultiplication  of  the element $\mathbbm{1}_{\mathbbm{g}}$ for every $\mathbbm{g}\in \mathbbm{G}$:
\begin{equation}\label{3.16}
\bigtriangleup(\mathbbm{1}_{\mathbbm{g}})=\sum_{\mathbbm{fh}=\mathbbm{g}}\mathbbm{1}_{\mathbbm{f}}\otimes \mathbbm{1}_{\mathbbm{h}}.
\end{equation}
Indeed, we have:
\begin{eqnarray*}
\sum_{\mathbbm{fh}=\mathbbm{g}}\mathbbm{1}_{\mathbbm{f}}\otimes \mathbbm{1}_{\mathbbm{h}}&=&\frac{1}{|\mathbbm{G}|^2}\sum_{\mathbbm{fh}=\mathbbm{g}}\sum_{\mathbbm{x}\in \mathbbm{G}}\chi_{\mathbbm{f}}(\mathbbm{x})\mathbbm{x}\otimes \sum_{\mathbbm{y}\in \mathbbm{G}}\chi_{\mathbbm{h}}(\mathbbm{y})\mathbbm{y} \\
&=&\frac{1}{|\mathbbm{G}|^2}\sum_{\mathbbm{x,y}\in \mathbbm{G}}[\sum_{\mathbbm{fh}=\mathbbm{g}}\chi_{\mathbbm{f}}(\mathbbm{x})\chi_{\mathbbm{h}}(\mathbbm{y})\mathbbm{x}\otimes \mathbbm{y}]\\
&=&\frac{1}{|\mathbbm{G}|}\sum_{\mathbbm{x}\in \mathbbm{G}}\chi_{\mathbbm{g}}(\mathbbm{x})\mathbbm{x}\otimes \mathbbm{x}\\
&=&\bigtriangleup(\mathbbm{1}_{\mathbbm{g}}),
\end{eqnarray*}
where the third identity follows from the equation:
\begin{equation*}
\sum_{\mathbbm{fh}=\mathbbm{g}}\chi_{\mathbbm{f}}(\mathbbm{x})\chi_{\mathbbm{h}}(\mathbbm{y})=\left\{
                                                                                                \begin{array}{ll}
                                                                                                  0, & \hbox{if $\mathbbm{x}\neq \mathbbm{y}$;} \\
                                                                                                  |\mathbbm{G}|\chi_{\mathbbm{g}}(\mathbbm{x}), &
                                                                                                 \hbox{if $\mathbbm{x}=\mathbbm{y}$.}
                                                                                                \end{array}
                                                                                              \right.
\end{equation*}
Hence, it yields:
\begin{eqnarray*}
\Phi_{\mathbbm{J}_{\underline{c}}}&=&(1\otimes \mathbbm{J}_{\underline{c}})(id\otimes \bigtriangleup)(\mathbbm{J}_{\underline{c}})(\bigtriangleup\otimes id)(\mathbbm{J}_{\underline{c}}^{-1})(\mathbbm{J}_{\underline{c}}\otimes 1)^{-1}\\
&=&\sum_{\mathbbm{f,g,h}\in \mathbbm{G}}\frac{J_{\underline{c}}(\mathbbm{g},\mathbbm{h})J_{\underline{c}}(\mathbbm{f},\mathbbm{gh})}
{J_{\underline{c}}(\mathbbm{f},\mathbbm{g})J_{\underline{c}}(\mathbbm{fg},\mathbbm{h})}\mathbbm{1}_{\mathbbm{f}}\otimes \mathbbm{1}_{\mathbbm{g}}\otimes \mathbbm{1}_{\mathbbm{h}}\\
&=&\sum_{\mathbbm{f,g,h}\in \mathbbm{G}}\partial(J_{\underline{c}})(\mathbbm{f,g,h})\mathbbm{1}_{\mathbbm{f}}\otimes \mathbbm{1}_{\mathbbm{g}}\otimes \mathbbm{1}_{\mathbbm{h}}.
\end{eqnarray*}
Now suppose $\mathbbm{f}=\prod_{i=1}^n\mathbbm{g}_i^{x_im_i+r_i},\ \mathbbm{g}=\prod_{i=1}^n\mathbbm{g}_i^{y_im_i+s_i},
\   \mathbbm{h}=\prod_{i=1}^n\mathbbm{g}_i^{z_im_i+t_i}$ for $0\leq x_i,y_i,z_i,r_i$, $s_i,t_i\leq m_i-1, 1\leq i\leq n.$  Let $f=\prod_{i=1}^ng_i^{r_i},g=\prod_{i=1}^ng_i^{s_i},\
h=\prod_{i=1}^ng_i^{t_i}$.   We compute  the element $\partial(J_{\underline{c}})(\mathbbm{f,g,h})$:
\begin{eqnarray*}
\partial(J_{\underline{c}})(\mathbbm{f,g,h})&=&
\prod_{i=1}^n\zeta_{\mathbbm{m}_i}^{c_it_i[\frac{r_i+s_i}{m_i}]m_i}\prod_{1\leq j<k\leq n}\zeta_{m_jm_k}^{c_{jk}t_k[\frac{r_j+s_j}{m_j}]m_j}\\
&=&\phi_{\underline{c}}(f,g,h).
\end{eqnarray*}
Applying  Lemma \ref{L3.2},  we obtain:
\begin{eqnarray*}
\Phi_{\mathbbm{J}_{\underline{c}}}&=&\sum_{0\leq x_l,y_l,z_l\leq m_l-1,1\leq l\leq n}\sum_{0\leq r_l,s_l,t_l\leq m_l-1,1\leq l\leq n}
\prod_{i=1}^n\zeta_{\mathbbm{m}_i}^{c_it_i[\frac{r_i+s_i}{m_i}]m_i}\prod_{1\leq j<k\leq n}\zeta_{m_jm_k}^{c_{jk}t_k[\frac{r_j+s_j}{m_j}]m_j}\\
&&\ \ \ \ \ \ \ \ \ \ \ \ \ \ \ \ \ \ \ \ \ \ \ \ \ \ \ \ \ \ \ \ \ \ \ \ \times \mathbbm{e}_{(\prod_{i=1}^n\mathbbm{g}_i^{x_im_i+r_i})}\otimes \mathbbm{e}_{(\prod_{i=1}^n\mathbbm{g}_i^{y_im_i+s_i})}\otimes \mathbbm{e}_{(\prod_{i=1}^n\mathbbm{g}_i^{z_im_i+t_i})}\\
&=&\sum_{0\leq r_l,s_l,t_l\leq m_l-1,1\leq l\leq n}\prod_{i=1}^n\zeta_{\mathbbm{m}_i}^{c_it_i[\frac{r_i+s_i}{m_i}]m_i}\prod_{1\leq j<k\leq n}\zeta_{m_jm_k}^{c_{jk}t_k[\frac{r_j+s_j}{m_j}]m_j} \\
&&\ \ \ \ \ \ \ \ \ \ \ \ \ \ \ \ \ \ \ \ \ \ \ \ \ \ \ \ \ \ \ \ \ \ \ \ \ \ \ \ \ \ \ \ \ \ \ \ \ \ \ \ \ \ \ \ \times 1_{(\prod_{i=1}^ng_i^{r_i})}\otimes 1_{(\prod_{i=1}^ng_i^{s_i})}\otimes 1_{(\prod_{i=1}^ng_i^{t_i})}\\
&=&\sum_{f,g,h\in G}\phi_{\underline{c}}(f,g,h)1_f\otimes 1_g\otimes 1_h,
\end{eqnarray*}
as desired.
\end{proof}

\begin{lemma}\label{L3.8}
$\mathbf{u}_\alpha(\mu)\in \k G$ for all $\alpha\in R^+.$
\end{lemma}
\begin{proof}
Since $R^+=\cup_{J\in \Omega}R_J^+,$  it suffices  to show  $\mathbf{u}_\alpha(\mu)\in \k G$ for any $\alpha\in R_J^+$ with a fixed $J\in \Omega.$ Suppose that $J=\{i_1,\cdots,i_\eta\}\subset I,$ and $\{\alpha_{i_1},\cdots,\alpha_{i_\eta}\}$ is the set of  the simple roots corresponding to the vertexes of $J.$ Let $w_J=s_{j_1}s_{j_2}\cdots s_{j_{P_J}}$ be the reduced presentation of the longest element of  the Weyl group $W_J$ in terms of simple reflections. Define
\begin{equation}\label{3.26}
\beta_{j_l}=s_{j_1}s_{j_2}\cdots s_{j_{l-i}}(\alpha_{j_{l}})\ \mathrm{for}\ 1\leq l\leq P_J
\end{equation}
and
\begin{equation}\label{3.27}
\underline{a}=a_1\beta_{j_i}+a_2\beta_{j_2}+\cdots+a_{P_J}\beta_{j_{P_J}}, a\in \mathbbm{N}^{P_J}.
\end{equation}
We  show that  $u^a\in \k G$ for each $a\in \mathbbm{N}^{P_J}$. Consequently,  it leads to  $\mathbf{u}_\alpha(\mu)\in \k G$ for all $\alpha\in R_J^+$. We will prove it by induction on $ht(\underline{a}).$

In case $ht(\underline{a})=1,$ then $\underline{a}$ is a simple root contained in $\{\alpha_{i_1},\cdots,\alpha_{i_\eta}\},$ say, $\alpha_{i_k}.$
By \eqref{2.17}, we have $u^a=\mu_a(1-h^a)=h_{a_{i_k}}^{N_J}.$  It follows from \ref{3.13}  that $u^a=\mu_a(1-h^a)\in \k G.$

Now assume  that $u^a\in \k G$ holds for all $\underline{a}\in \mathbbm{N}^{P_J}$ such that $ht(\underline{a})<l$.  Let $a\in \mathbbm{N}^{P_J}$ such that $ht(\underline{a})=l.$ If $\underline{a}=\beta_{j_s}$ for some $1\leq s\leq P_J,$ then
\begin{equation*}
u^a=\mu_a(1-h^a)+\sum_{b,c\neq 0,\underline{b}+\underline{c}=\underline{a}}t^a_{b,c}\mu_bu^c,
\end{equation*} and $h^a=h_{\beta_{j_s}}^{N_J}.$ From \ref{3.13} we see that  the part $\mu_a(1-h^a)\in \k G$.  The fact that second part $\sum_{b,c\neq 0,\underline{b}+\underline{c}=\underline{a}}t^a_{b,c}\mu_bu^c$ belongs to $\k G$ follows from the induction assumption. If $\underline{a}\neq\beta_{j_s}$ for any $1\leq s\leq P_J,$ then $a=(a_1,a_2,\cdots,a_s,\cdots,0)$ with $ a_s>0$ for some $ 1\leq s\leq P_J.$ Let $$e_s=(\underbrace{0,\cdots, 1}_s,0,\cdots,0).$$
By Proposition \ref{P2.9}, we have $u^a=u^{a-e_s}u^{e_s}.$ Since  both heights of $a-e_s$ and $e_s$ are  less than  $l$,   the elements $u^{a-e_s},u^{e_s}$ belong to $\k G$ by the induction assumption. This implies  that $u^a\in \k G.$
\end{proof}

Now we denote by  $A(H,\underline{c})$  the subalgebra of $H^{\mathbbm{J}_{\underline{c}}}$ generated by $G$ and $\{X_1,\cdots,X_\theta\}.$  We are going to show that $A(H,\underline{c})$ is  the desired quasi-Hopf algebra if we choose an approriate element $\underline{c}$.  We first describe the defining relations  of the generators of $A(H,\underline{c})$.

\begin{proposition}\label{p3.11} The algebra $A(H,\underline{c})$ can be presented by the generators $G$ and $\{X_1,\cdots,X_\theta\}$ and  the following relations:
\begin{eqnarray}
&&gX_ig^{-1}=\chi_i(g)X_i,\ \mathrm{for\ all}\ 1\leq i\leq \theta, g\in G,\\ 
&&ad_C(X_i)^{1-a_{ij}}(X_j)=0,\ \mathrm{for\ all}\ i\neq j, i\sim j, \\
&&ad_C(X_i)(X_j)=\lambda_{ij}(1-h_ih_j),\ \mathrm{for\ all}\ i<j, i\nsim j,\\ 
&&X_\alpha^{N_J}=\mathbf{u}_{\alpha}(\mu), \ \mathrm{for\ all}\ \alpha\in R_J^+, J\in \Omega. 
\end{eqnarray}
Moreover, $A(H,\underline{c})$ has a basis of the form
$$X_{\beta_1}^{x_1}X_{\beta_2}^{x_2}\cdots X_{\beta_P}^{x_P}g,\ g\in G,\ 0\leq x_i\leq N_J,\ \beta_i\in R_J^+.$$
\end{proposition}
\begin{proof}
Since the relations (3.34)-(3.35) hold in $H$ for the group $\mathbbm{G}$ and the generators $X_1, \cdots X_\theta$,  they hold as well for the subgroup $G$ and $X_1, \cdots, X_\theta$.  Relation  (3.37) follows from Lemma \ref{L3.8}. For the relation (3.36),  it is enought to show that  the elements   $\lambda_{ij}(1-h_ih_j)$  fall in $\k G$.  But this is true because of \eqref{3.12}.
The last part of the proposition follows from \cite[Theorem 3.3]{AS2} and Relation \eqref{3.22}.
\end{proof}
The algebra  $A(H,\underline{c})$ is apparently not a Hopf subalgebra of $H$. However, it is a quasi-Hopf subalgebra of  some twist of $H$.

\begin{proposition}\label{P3.10}
$A(H,\underline{c})$ is a quasi-Hopf subalgebra of $H^{\mathbbm{J}_{\underline{c}}}$ if and only if $\underline{c}\in \Gamma(\mathfrak{D}).$
\end{proposition}
\begin{proof}
$\Leftarrow.$ First of all, we show that $A(H,\underline{c})$ is closed under the comultiplication $\bigtriangleup_{\mathbbm{J}_{\underline{c}}}$ of  $H^{\mathbbm{J}_{\underline{c}}}.$
It is obvious that $\bigtriangleup_{\mathbbm{J}_{\underline{c}}}(g)=\mathbbm{J}_{\underline{c}}(g\otimes g)\mathbbm{J}^{-1}_{\underline{c}}=g\otimes g$ for any $g\in G\subset \mathbbm{G}$ since $\mathbbm{G}$ is abelian.  It remains to show that
$\bigtriangleup_{\mathbbm{J}_{\underline{c}}}(X_i)\in A(H,\underline{c})\otimes A(H,\underline{c})$ for $1\leq i\leq \theta.$ By Lemma \ref{L3.1} and \ref{L3.5}, we have
\begin{eqnarray*}
\bigtriangleup_{\mathbbm{J}_{\underline{c}}}(X_i)&=&\mathbbm{J}_{\underline{c}}(X_i\otimes 1+h_i\otimes X_i)\mathbbm{J}^{-1}_{\underline{c}}  \\
&=&(\sum_{\mathbbm{f,g}\in\mathbbm{G}}J_{\underline{c}}(\mathbbm{f,g})\mathbbm{1}_{\mathbbm{f}}\otimes \mathbbm{1}_{\mathbbm{g}})(X_i\otimes 1+h_i\otimes X_i)(\sum_{\mathbbm{f,g}\in\mathbbm{G}}J^{-1}_{\underline{c}}(\mathbbm{f,g})\mathbbm{1}_{\mathbbm{f}}\otimes \mathbbm{1}_{\mathbbm{g}})\notag\\
&=&\sum_{\mathbbm{f,g}\in\mathbbm{G}}[\frac{J_{\underline{c}}(\mathbbm{f}\eta_i^{-1},\mathbbm{g})}{J_{\underline{c}}(\mathbbm{f,g})}
X_i\mathbbm{1}_{\mathbbm{f}}\otimes \mathbbm{1}_{\mathbbm{g}}+\frac{J_{\underline{c}}(\mathbbm{f},\mathbbm{g}\eta_i^{-1})}{J_{\underline{c}}(\mathbbm{f,g})}
\chi_{\mathbbm{f}}^{-1}(h_i)\mathbbm{1}_{\mathbbm{f}}\otimes X_i\mathbbm{1}_{\mathbbm{g}}].\notag
\end{eqnarray*}
Suppose $\mathbbm{f}=\prod_{j=1}^{n}\mathbbm{g}_j^{x_jm_j+k_j}, \mathbbm{g}=\prod_{j=1}^{n}\mathbbm{g}_j^{y_jm_j+l_j}$ for $0\leq x_j,y_j,k_j,l_j\leq m_j-1, 1\leq j\leq n.$ Let $(k_j-r_{ij})'$ be the remainder of $(k_j-r_{ij})$ divided by $m_j$ for $1\leq 1\leq j\leq n,$ and define
\begin{equation}
\psi_{ij}=\left\{
           \begin{array}{ll}
            (k_j-r_{ij})'-(x_jm_j+k_j-r_{ij}), & \hbox{if $x_jm_j+k_j-r_{ij}\geq 0$;} \\
            (k_j-r_{ij})'-\mathbbm{m}_j-(x_jm_j+k_j-r_{ij}), & \hbox{if $x_jm_j+k_j-r_{ij}<0$.}
           \end{array}
         \right.
\end{equation} Then by \eqref{3.9} we have
\begin{eqnarray}
\frac{J_{\underline{c}}(\mathbbm{f}\eta_i^{-1},\mathbbm{g})}{J_{\underline{c}}(\mathbbm{f,g})}&=&\frac{\prod_{j=1}^{n}
\zeta_{\mathbbm{m}_j}^{c_{j}l_{j}\psi_{ij}}
\prod_{1\leq s<t\leq n}\zeta_{m_sm_t}^{c_{st}l_{t}\psi_{is}}}{\prod_{j=1}^{n}
\zeta_{\mathbbm{m}_j}^{-c_{j}l_{j}x_jm_j}
\prod_{1\leq s<t\leq n}\zeta_{m_sm_t}^{-c_{st}l_{t}x_sm_s}}\\
&=&\prod_{j=1}^{n}
\zeta_{\mathbbm{m}_j}^{c_{j}l_{j}((k_j-r_{ij})'-(k_j-r_{ij}))}
\prod_{1\leq s<t\leq n}\zeta_{m_sm_t}^{c_{st}l_{t}((k_s-r_{is})'-(k_s-r_{is}))}\notag\\
&=&\prod_{j=1}^{n}
\zeta_{\mathbbm{m}_j}^{c_{j}l_{j}\varphi_{ij}}
\prod_{1\leq s<t\leq n}\zeta_{m_sm_t}^{c_{st}l_{t}\varphi_{is}}.\notag
\end{eqnarray}
Similarly,  by \eqref{3.7} \and \eqref{3.8}, we obtain
\begin{eqnarray}
\frac{J_{\underline{c}}(\mathbbm{f},\mathbbm{g}\eta_i^{-1})}{J_{\underline{c}}(\mathbbm{f,g})}\chi_{\mathbbm{f}}^{-1}(h_i)&=&
J_{\underline{c}}(\mathbbm{f},\eta_i^{-1})\chi_{\mathbbm{f}}^{-1}(h_i)\\
&=&\prod_{j=1}^{n}\zeta_{\mathbbm{m}_j}^{c_{j}r_{ij}x_jm_j)}
\prod_{1\leq s<t\leq n}\zeta_{m_sm_t}^{c_{st}r_{it}(x_sm_s)}\prod_{j=1}^{n}\zeta_{\mathbbm{m}_j}^{-(x_jm_j+k_j)s_{ij}}\notag\\
&=&\prod_{j=1}^{n}\zeta_{\mathbbm{m}_j}^{(c_{j}r_{ij}-s_{ij})x_jm_j)}\prod_{1\leq s<t\leq n}\zeta_{m_sm_t}^{c_{st}r_{it}(x_sm_s)}
\prod_{j=1}^{n}\zeta_{\mathbbm{m}_j}^{-k_js_{ij}}  \notag\\
&=&\prod_{j=1}^{n}\zeta_{\mathbbm{m}_j}^{-k_js_{ij}}.\notag
\end{eqnarray}
Hence,
\begin{eqnarray}
\bigtriangleup_{\mathbbm{J}_{\underline{c}}}(X_i)&=&\sum_{0\leq x_r,y_r,k_r,l_r\leq m_r-1, 1\leq r\leq n}[\prod_{j=1}^{n}
\zeta_{\mathbbm{m}_j}^{c_{j}l_{j}\varphi_{ij}}\prod_{1\leq s<t\leq n}\zeta_{m_sm_t}^{c_{st}l_{t}\varphi_{is}}\label{3.33}\\
&&\ \ \ \ \ \ \ \ \ \ \ \ \ \ \ \ \ \ \ \ \ \ \ \ \ \ \ \ \ \ \ \   \times X_i\mathbbm{1}_{(\prod_{j=1}^{n}\mathbbm{g}_j^{x_jm_j+k_j})}\otimes \mathbbm{1}_{(\prod_{j=1}^{n}\mathbbm{g}_j^{y_jm_j+l_j})}\notag\\
&&\ \ \ \ \ \ \   +\prod_{j=1}^{n}\zeta_{\mathbbm{m}_j}^{-k_js_{ij}}\mathbbm{1}_{(\prod_{j=1}^{n}\mathbbm{g}_j^{x_jm_j+k_j})}\otimes X_i\mathbbm{1}_{(\prod_{j=1}^{n}\mathbbm{g}_j^{y_jm_j+l_j})}]\notag\\
&=&\sum_{0\leq k_r,l_r\leq m_r-1, 1\leq r\leq n}[\prod_{j=1}^{n}
\zeta_{\mathbbm{m}_j}^{c_{j}l_{j}\varphi_{ij}}\prod_{1\leq s<t\leq n}\zeta_{m_sm_t}^{c_{st}l_{t}\varphi_{is}}\notag\\
&&\ \ \ \ \ \ \ \ \ \ \ \ \ \ \ \ \ \ \ \ \ \ \ \ \ \ \ \ \ \ \ \ \ \ \ \ \ \ \ \ \ \ \ \ \ \ \ \ \times X_i1_{(\prod_{j=1}^{n}g_j^{k_j})}\otimes 1_{(\prod_{j=1}^{n}g_j^{l_j})}\notag\\
&&\ \ \ \ \ \ \ \  +\prod_{j=1}^{n}\zeta_{\mathbbm{m}_j}^{-k_js_{ij}}1_{(\prod_{j=1}^{n}g_j^{k_j})}\otimes X_i1_{(\prod_{j=1}^{n}g_j^{l_j})}].\notag \\
&=&\sum_{f,g\in G}\Psi(f,g)X_i1_f\otimes 1_g+\sum_{f\in G}\Theta_l(f)1_f\otimes X_i.\notag                                                                                                       \end{eqnarray}
The second equality follows from Lemma \ref{L3.2}. So we have proved that $\bigtriangleup_{\mathbbm{J}_{\underline{c}}}(X_i)\in A(H,\underline{c})\otimes A(H,\underline{c}),$ hence
$A(H,\underline{c})$ is closed under the comultiplication $\bigtriangleup_{\mathbbm{J}_{\underline{c}}}$ of  $H^{\mathbbm{J}_{\underline{c}}}.$

Next we will show that $(S_{\mathbbm{J}_{\underline{c}}}|_{A(H,\underline{c})},\beta_{\mathbbm{J}_{\underline{c}}}\alpha_{\mathbbm{J}_{\underline{c}}},1)$ is a antipode of $A(H,\underline{c}).$  For all $\mathbbm{g}\in \mathbbm{G},$ we have $$S(\mathbbm{1}_\mathbbm{g})=S(\frac{1}{|\mathbbm{G}|}\sum_{\mathbbm{h}\in\mathbbm{G}}\chi_{\mathbbm{g}}(\mathbbm{h})\mathbbm{h})=
\frac{1}{|\mathbbm{G}|}\sum_{\mathbbm{h}\in\mathbbm{G}}\chi_{\mathbbm{g}}(\mathbbm{h})\mathbbm{h}^{-1}=\mathbbm{1}_{\mathbbm{g}^{-1}}.$$
So we obtain
\begin{eqnarray}
&&\alpha_{\mathbbm{J}_{\underline{c}}}=\sum_{\mathbbm{f,g}\in \mathbbm{G}}J^{-1}_{\underline{c}}(\mathbbm{f,g})\mathbbm{1}_{\mathbbm{f}^{-1}} \mathbbm{1}_{\mathbbm{g}}=\sum_{\mathbbm{g}\in \mathbbm{G}}J^{-1}_{\underline{c}}(\mathbbm{g}^{-1},\mathbbm{g})\mathbbm{1}_{\mathbbm{g}},\\
&&\beta_{\mathbbm{J}_{\underline{c}}}=\sum_{\mathbbm{f,g}\in \mathbbm{G}}J_{\underline{c}}(\mathbbm{f,g})\mathbbm{1}_{\mathbbm{f}} \mathbbm{1}_{\mathbbm{g}^{-1}}=\sum_{\mathbbm{g}\in \mathbbm{G}}J_{\underline{c}}(\mathbbm{g},\mathbbm{g}^{-1})\mathbbm{1}_{\mathbbm{g}}.\label{3.40}
\end{eqnarray}
It is obvious that $\beta_{\mathbbm{J}_{\underline{c}}}$ is invertible with inverse $\sum_{\mathbbm{g}\in \mathbbm{G}}J^{-1}_{\underline{c}}(\mathbbm{g},\mathbbm{g}^{-1})\mathbbm{1}_{\mathbbm{g}},$ and we have:
\begin{eqnarray}
&&\beta_{\mathbbm{J}_{\underline{c}}}\alpha_{\mathbbm{J}_{\underline{c}}}=\sum_{\mathbbm{g}\in \mathbbm{G}}\frac{J_{\underline{c}}(\mathbbm{g},\mathbbm{g}^{-1})}{J_{\underline{c}}(\mathbbm{g}^{-1},\mathbbm{g})}
\mathbbm{1}_{\mathbbm{g}}\label{3.36}\\
&&=\sum_{0\leq k_i,l_i\leq m_i-1,1\leq i\leq n}\frac{\prod_{i=1}^n\zeta_{\mathbbm{m}_i}^{c_il_ik_im_i}\prod_{1\leq s<t\leq n}\zeta_{m_sm_t}^{c_{st}l_tk_sm_s}}{\prod_{i=1}^n\zeta_{\mathbbm{m}_i}^{c_il_i(k_i+1)m_i}\prod_{1\leq s<t\leq n}\zeta_{m_sm_t}^{c_{st}l_t(k_s+1)m_s}}\mathbbm{1}_{(\prod_{i=1}^n\mathbbm{g}_i^{k_im_i+l_i})}\notag\\
&&=\sum_{0\leq k_i,l_i\leq m_i-1,1\leq i\leq n}\prod_{i=1}^n\zeta_{\mathbbm{m}_i}^{-c_il_im_i}\prod_{1\leq s<t\leq n}\zeta_{m_sm_t}^{-c_{st}l_tm_s}\mathbbm{1}_{(\prod_{i=1}^n\mathbbm{g}_i^{k_im_i+l_i})}\notag\\
&&=\sum_{0\leq l_i\leq m_i-1,1\leq i\leq n}\prod_{i=1}^n\zeta_{\mathbbm{m}_i}^{-c_il_im_i}\prod_{1\leq s<t\leq n}\zeta_{m_sm_t}^{-c_{st}l_tm_s}
1_{(\prod_{i=1}^ng_i^{l_i})}\notag\\
&&=\sum_{g\in G}\Upsilon(g)1_g.\notag
\end{eqnarray}  Here the second identity follows from \eqref{3.9}, and the fourth identity follows from Lemma \ref{L3.2}. Hence we have showed $\beta_{\mathbbm{J}_{\underline{c}}}\alpha_{\mathbbm{J}_{\underline{c}}} \in A(H,\underline{c}).$
Next we will show that $S_{\mathbbm{J}_{\underline{c}}}$  preserve $A(H,\underline{c}).$ It is obvious that
$S_{\mathbbm{J}_{\underline{c}}}(g)=g^{-1}$ for all $g\in G.$ For each $X_i, 1\leq i\leq \theta$ we have
\begin{eqnarray}
S_{\mathbbm{J}_{\underline{c}}}(X_i)&=&\beta_{\mathbbm{J}_{\underline{c}}}S(X_i)\beta^{-1}_{\mathbbm{J}_{\underline{c}}}=
-\beta_{\mathbbm{J}_{\underline{c}}}(h_i^{-1}X_i)\beta^{-1}_{\mathbbm{J}_{\underline{c}}}\label{3.37}\\
&=&-\sum_{\mathbbm{f,g}\in \mathbbm{G}}\frac{J_{\underline{c}}(\mathbbm{f},\mathbbm{f}^{-1})}{J_{\underline{c}}(\mathbbm{g},\mathbbm{g}^{-1})}
\chi_{\mathbbm{f}}(h_i)\mathbbm{1}_{\mathbbm{f}}X_i\mathbbm{1}_{\mathbbm{g}}\notag\\
&=&-\sum_{\mathbbm{g}\in \mathbbm{G}}\frac{J_{\underline{c}}(\mathbbm{g}\eta_i^{-1},\mathbbm{g}^{-1}\eta_i)}{J_{\underline{c}}(\mathbbm{g},\mathbbm{g}^{-1})}
\chi_{\mathbbm{g}\eta_i^{-1}}(h_i)X_i\mathbbm{1}_{\mathbbm{g}}\notag\\
&=&\sum_{0\leq x_l,k_l\leq m_l-1,1\leq l\leq n}\prod_{j=1}^n\zeta_{\mathbbm{m}_j}^{-c_j(k_j-r_{ij})[(k_j-r_{ij})'-(k_j-r_{ij})]} \notag\\
&&\times  \ \ \ \ \ \ \prod_{1\leq s<t\leq n}\zeta_{\mathbbm{m}_s\mathbbm{m}_t}^{-c_{st}(k_t-r_{it})[(k_s-r_{is})'-(k_s-r_{is})]} X_i\mathbbm{1}_{(\prod_{l=1}^n\mathbbm{g}_l^{x_lm_l+k_l})}\notag\\
&=&\sum_{0\leq k_l\leq m_l-1,1\leq l\leq n}\prod_{j=1}^n\zeta_{\mathbbm{m}_j}^{-c_j(k_j-r_{ij})[(k_j-r_{ij})'-(k_j-r_{ij})]} \notag\\
&&\times  \ \ \ \ \ \ \prod_{1\leq s<t\leq n}\zeta_{\mathbbm{m}_s\mathbbm{m}_t}^{-c_{st}(k_t-r_{it})[(k_s-r_{is})'-(k_s-r_{is})]} X_i1_{(\prod_{l=1}^ng_l^{k_l})}\notag\\
&=&\sum_{g\in G}\mathcal{F}_i(g)X_i1_g,\notag
\end{eqnarray}
where the third identity follows from Lemma \ref{L3.1} and \ref{L3.5}; the fourth identity follows from \eqref{3.7}-\eqref{3.9}; and the fifth identity follows from Lemma \ref{L3.2}.
So $S_{\mathbbm{J}_{\underline{c}}}(X_i)\in A(H, \underline{c}),$ and $S_{\mathbbm{J}_{\underline{c}}}$ preserves $A(H,\underline{c})$ because $S_{\mathbbm{J}_{\underline{c}}}$ is an anti-algebra morphism and $A(H,\underline{c})$ is generated by $G$ and $\{X_1,\cdots,X_\theta\}.$

$\Rightarrow.$  We omit the detailed computation, and  point out that $\bigtriangleup_{\mathbbm{J}_{\underline{c}}}(X_i)\in A(H,\underline{c})\otimes A(H,\underline{c})$ implies \eqref{3.7}-\eqref{3.9}, and $S_{\mathbbm{J}_{\underline{c}}}(X_i)\in A(H,\underline{c})$ implies \eqref{3.7}-\eqref{3.9}. Hence $\underline{c}$ must be contained in $\Gamma(\mathfrak{D}).$
\end{proof}

$\mathbf{Proof\ of\ Theorem\ \ref{T3.4}}.$ Let $\underline{c}\in \Gamma(\mathfrak{D})$.  By Proposition \ref{p3.11}, we know that $A(H,\underline{c})$ is identical to $u(\mathfrak{D},\lambda,\mu,\Phi_{\underline{c}})$ as algebras. By Proposition \ref{P3.10}, $A(H,\underline{c})$ is a quasi-Hopf algebra.  Thus,  $u(\mathfrak{D},\lambda,\mu,\Phi_{\underline{c}})$ is also a quasi-Hopf algebra with the comultiplication $\bigtriangleup$ satisfying $\bigtriangleup(g)=g\otimes g$ and \eqref{3.33}.  The antipode $(\mathcal{S},\mathcal{\alpha},1)$ is determined by \eqref{3.36} and \eqref{3.37},  and  the associator $\Phi$ is given by \eqref{3.22}. Therefore, we have proved the theorem. \hfill $\Box$

\subsection{Examples of quasi-Hopf algebras of Cartan type}
In this subsection, we will give some examples of quasi-Hopf algebras of Cartan type. We make a convention that the comultiplications,  the associators, and the antipodes of the quasi-Hopf algebras  in those examples below can be written in the forms as listed in \eqref{3.18}-\eqref{3.20}, and hence will be omitted.

\begin{example}{\bf Basic quasi-Hopf algebras over cyclic groups.}
Let $G=\Z_m=\langle g\rangle$. In this case,  $\mathbbm{G}=\Z_{m^2}=\langle \mathbbm{g}\rangle,$ and $G$ is identical to the  subgroup $\langle \mathbbm{g}^m\rangle$ of $\mathbbm{G}$, see \eqref{3.3}. Let $\mathfrak{D}=\mathfrak{D}(\mathbbm{G},(h_i)_{1\leq i\leq \theta},(\chi_i)_{1\leq i\leq \theta}, A)$ be a datum of finite Cartan type, where $$h_i=\mathbbm{g}^{s_i}, \chi_i(\mathbbm{g})=\zeta_{m^2}^{r_i}$$ for some $0< s_i,r_i< m^2, 1\leq i\leq \theta.$
Denote by $H$ the AS-Hopf algebra $u(\mathfrak{D},0,0)$, and  let $s$ be a number satisfying $0<s<m$ and $sr_i\equiv s_i \mod m$ for all $1\leq i\leq \theta,$  and $\underline{c}=\{s\}.$ Then we get a quasi-Hopf algebra $u(\mathfrak{D},0,0,\Phi_{\underline{c}})=A(H,s),$ see \cite[Proposition 3.1.1]{A} for the definition of $A(H,s).$  According to \cite[Theorem 3.4.1]{A}, any nonsemisimple, genuine basic graded quasi-Hopf algebra over a cyclic group with dimension not divisible by $2$ and $3$ must be twist equivalent to some $u(\mathfrak{D},0,0,\Phi_{\underline{c}})$.

 Let $m=p$ and $\mathfrak{D}=\mathfrak{D}(\mathbbm{G},\mathbbm{g},\chi, A_1)$ such that $\chi(\mathbbm{g})=\zeta_{p^2}.$ Then it is obvious that $s=1$, $\underline{c}=\{1\}$, and we get a quasi-Hopf algebra $u(\mathfrak{D},0,0,\Phi_{\underline{c}})$ generated by $G$ and $X$ with the relations $$gXg^{-1}=\zeta_pX,\  X^{p^2}=0.$$
According to the classification of pointed Hopf algebras of dimension $p^3$ in \cite{AS1}, we know that $u(\mathfrak{D},0,0,\Phi_{\underline{c}})$ does not admit a pointed Hopf algebra structure.
\end{example}

Next we will construct a few more non-radically graded quasi-Hopf algebras of rank $2.$ First we give an example of a quasi-Hopf algebra associated to $A_1\times A_1,$ and then present some examples of quasi-Hopf algebras associated to $A_2,B_2, G_2.$

\begin{example}{\bf The quasi-version of $u_q(sl_2)$.}\label{E3.14} Let $N>2$ and $d$ be two positive odd numbers, and  $G=\Z_m=\langle g\rangle$,   $\mathbbm{G}=\Z_{m^2}=\langle \mathbbm{g}\rangle$, where $m=Nd.$   As  usual, $G$ is viewed as the subgroup of $\mathbbm{G}$.
Let $\mathfrak{D}=\mathfrak{D}(\mathbbm{G},(h_1,h_2), (\chi_1,\chi_2),A_1\times A_1),$ where $$h_1=h_2=\mathbbm{g}^{m},\chi_1(\mathbbm{g})=\zeta_{m^2}^{2d}, \chi_2(\mathbbm{g})=\zeta_{m^2}^{-2d}.$$ It is  easy to verify that $\mathfrak{D}$ is a datum of Cartan type. Since $\Gamma(\mathfrak{D})$ is the set of numbers $0\leq c\leq m-1$ satisfying
\begin{eqnarray*}
&&m\equiv 2cd\ \ \ \ \mod m,\\
&&m\equiv -2cd\ \ \ \ \mod m.
\end{eqnarray*}
 Both equations are equivalent with $N|c$ since $m=Nd$  and $N$ is odd.

In this case, it is clear that $\Gamma(\mathfrak{D})=\{c=kN|0\leq k<d\}$. Let $q=\zeta_{m^2}^{md}$.   By $H_c$ we denote the algebra generated by $g,X_1,X_2$ subject to the relations as follows:
\begin{eqnarray*}
&& gX_1g^{-1}=q^2X_1, gX_2g^{-1}=q^{-2}X_2,\\
&& X_1X_2-q^{-2}X_2X_1=\lambda(1-g^2),\\
&& X_1^N=X_2^N=0.
\end{eqnarray*}
Let $E=X_1, F=X_2g^{-1}$,  and $\lambda= q^{-1}-q,$ then we can see that $H_c$ is generated by $g, E,F$ satisfying the relations:
\begin{eqnarray*}
&& gEg^{-1}=q^2X_1, qFg^{-1}=q^{-2}X_2,\\
&& EF-FE=\frac{g-g^{-1}}{q-q^{-1}},\\
&& E^N=F^N=0.
\end{eqnarray*}

When $d=1,$ we have $c=0$ and $H_c\cong u_q(sl_2).$ When $d>1,$  we know that $\Gamma(\mathfrak{D})$ has nonzero elements.  If $c=0,$ we have $H_c\cong u(\mathfrak{D}',\lambda,0),$ where $\mathfrak{D}'=\mathfrak{D}(G,(g,g),(\chi'_1,\chi'_2),A_1\times A_1)$ and $\chi'_1(g)=q^2,\chi_2(g)=q^{-2}.$  If $c\in \Gamma(\mathfrak{D})$ is nonzero, then we have $H_c= u(\mathfrak{D},\lambda,0,\Phi_{c}).$  Because of this fact, $u(\mathfrak{D},\lambda,0,\Phi_{c})$ can be viewed as the quasi-version of $u_q(sl_2)$, and is called {\bf a small quasi-quantum group}.  More small quasi-quantum groups will be studied in Section 5.
\end{example}

The above example provides us some non-radically graded quasi-Hopf algebras associated to $A_1\times A_1.$  We point out that there is a similar notion of a quasi-version of $u_q(sl_2)$ in \cite{Liu}, where Liu defined a quasi-Hopf analogue of $u_q(sl_2)$ as a quantum double of a quasi-Hopf algebra associated to $A_1$. It is obvious that these two definitions are different, since the dimension of a quasi-version of $u_q(sl_2)$ is not a square in general. Hence,  it should not be a quantum double. In order to construct non-radically graded examples of type $A_2,B_2,G_2,$ we need the following well-known proposition from number theory.

\begin{proposition}\label{P3.15}
Let $a, b, n$ be nonzero integers.
Then the equation $ax\equiv b \mod n$ has solutions if and only if $(a,n)|b.$ Moreover, if there exists a  solution, then it is unique up to modulo $\frac{n}{(a,n)}.$
\end{proposition}

\begin{example}{\bf Quasi-Hopf algebras associated to $A_2,B_2,G_2.$}\label{E3.16} Let $A$ be a Cartan matrix of type $A_2,$ $B_2$ or $G_2.$ Suppose that $m,n,d$ are positive odd numbers such that $(m,n)=(m,d)=(n,d)=1.$ In addition,  in case $A$ is of type $G_2,$ we will assume that the three numbers $m,n$ and $d$  are not divisible by $3.$
Let $G=\langle g_1\rangle \times \langle g_2\rangle \times \langle g_3\rangle \times \langle g_4\rangle,$ where $|g_1|=md, |g_2|=nd, |g_3|=md^2$, and $|g_4|=nd^2.$ The group $\mathbbm{G}$ and the generators $\mathbbm{g}_i, 1\leq i\leq 4$ are defined in a similar manner as before. Let $a,b,k$ be the numbers listed in the Table \ref{tab.1}.
Define $\mathfrak{D}=\mathfrak{D}(\mathbbm{G},(h_1,h_2),(\chi_1,\chi_2),A),$ where
\begin{eqnarray*}
&&h_1=(\mathbbm{g}_1\mathbbm{g}_2)^{mn}, \ h_2=(\mathbbm{g}_1\mathbbm{g}_2\mathbbm{g}_3\mathbbm{g}_4)^{mn}\\
&&\chi_1(\mathbbm{g}_1)=\zeta_{d^2}^a, \  \chi_1(\mathbbm{g}_2)=\zeta_{d^2}^a, \ \chi_1(\mathbbm{g}_3)=\zeta_{d^2}^b, \ \chi_1(\mathbbm{g}_4)=\zeta_{d^2}^b,\\
&&\chi_2(\mathbbm{g}_1)=\zeta_{d^2}^a, \ \chi_2(\mathbbm{g}_2)=\zeta_{d^2}^a, \ \chi_2(\mathbbm{g}_3)=\zeta_{d^4}, \ \chi_2(\mathbbm{g}_4)=\zeta_{d^4}^{kd^2-1}.
\end{eqnarray*}
Let $A=(a_{ij})_{1\leq i,j\leq 2}$ such that $a_{12}\leq a_{21}.$ Then one can easily show that
$$\chi_1(h_2)\chi_2(h_1)=\chi_1(h_1)^{a_{12}}=\chi_2(h_2)^{a_{21}}.$$
Hence $\mathfrak{D}$ is a datum of Cartan type.

Next we will show that $\Gamma(\mathfrak{D})$ contains nonzero elements. By definition, $\Gamma(\mathfrak{D})$ is the set of families $\underline{c}=(c_i,c_{st})_{1\leq i\leq 4,1\leq s<t\leq 4}$ satisfying \eqref{3.7}-\eqref{3.9}. So we only need to show that Equations \eqref{3.7} have nonzero solutions $\{c_1,c_2,c_3,c_4\},$ since \eqref{3.8}-\eqref{3.9} always have solutions. Equations \eqref{3.7}  are equivalent to:
\begin{eqnarray}
&& mn\equiv c_1m^2a \ \ \mod md;\label{3.44}\\
&& mn\equiv c_2n^2a \ \ \mod nd;\label{3.45}\\
&& 0\equiv c_3(bm^2d^2)\ \ \mod md^2, mn\equiv c_3m^2\ \ \mod md^2;\label{3.46}\\
&& 0\equiv c_4(bm^2d^2)\ \ \mod nd^2, mn\equiv c_4n^2(kd^2-1)\ \ \mod nd^2.\label{3.47}
\end{eqnarray}
By Proposition \ref{P3.15}, these equations have solutions. It is obvious that any solution $(c_1,c_2,c_3,c_4)$ of Equations \eqref{3.44}-\eqref{3.47} should not be zero since $(n,d)=1$. Hence,  $\Gamma(\mathfrak{D})$ contains nonzero elements.

At last, we will show that there exists a family of nonzero modified root vector parameters $\mu$ for $\mathfrak{D}.$ Note that $A$ is connected, so $\lambda$ must be zero.
Since $N=|\chi_i(h_i)|$ for $i=1,2,$ (see \eqref{2.9} for definition), it is obvious that $N=|\zeta_{d^2}^{2amn}|=d^2.$ Let $\alpha_i$ be the simple root corresponding to $X_i,$ $1\leq i\leq 2.$ We have $h_1^{N}=(\mathbbm{g}_1\mathbbm{g}_2)^{d^2mn}=g_1^{dn}g_2^{dm}\in G,$ $h_1^{N}\neq 1,$ and $\chi_1^{N}=\chi_1^{d^2}=1.$ So $\mu_{\alpha_1}$ is a nonzero parameter. Thus,  there exists a family $\mu$ of nonzero modified root vector parameters  for $\mathfrak{D}.$ The quasi-Hopf algebra  $u(\mathfrak{D},0,\mu,A)$ is a nonradically graded quasi-Hopf algebra associated to $A.$ In Section 6, we will show that these quasi-Hopf algebras are genuine.
\end{example}

{\setlength{\unitlength}{1mm}
\begin{table}[t]\centering
\caption{$a,b,k$ associated to $A_2,B_2,G_2.$.} \label{tab.1}
\vspace{1mm}
\begin{tabular}{r|l|l}
\hline
  &\text{$a,b,k$ associated to $A$}& \text{Cartan matrix $A$}\\
\hline
\hline
  1.  & $a=1,\ b=-3,\ k=0 $ &\ \ \ \ \ $A_2$\\
\hline
  2. & $a=1,\ b=-3,\ k=-1 $& \ \ \ \ \  $B_2$\\
  \hline
  3. & $a=3,\ b=-6,\ k=-4 $& \ \ \ \ \ $G_2$\\
  \hline
 \end{tabular}
\end{table}}

\begin{remark}
Consider the subalgebra of $u(\mathfrak{D},0,\mu,A)$  generated by $g_1,g_2, X_1$ in Example \ref{E3.16}.  This algebra is a quasi-Hopf algebra of rank one with a nontrivial root vector relation. Hence, it is  nonradically graded.  We make a convention that if $u(\mathfrak{D},\lambda,\mu,A)$ is a quasi-Hopf algebra over a cyclic group, then the root vector relation must be trivial, i.e. $\mu=0$. For this reason, we can only construct quasi-Hopf algebras with nontrivial root vector relations over noncyclic groups.
\end{remark}

\section{Radically graded quasi-Hopf algebras of Cartan type} In this section, we study  radically graded quasi-Hopf algebras of Cartan type. We show that all the radically graded quasi-Hopf algebras  given in Theorem \ref{T3.4} are genuine quasi-Hopf algebras, which leads  to some interesting classification results.

\subsection{Radically graded quasi-Hopf algebras of Cartan type are genuine} In general, it is very difficult to determine whether a nonradically graded quasi-Hopf algebra is genuine or not. However, for radically graded quasi-Hopf algebras, we have the following proposition.

\begin{proposition}\label{P4.1}
Suppose that $H=\oplus_{i\geq 0}H[i]$ is a finite-dimensional radically graded quasi-Hopf algebra. Then $H$ is genuine if and only if $H[0]$ is a genuine quasi-Hopf algebra.
\end{proposition}
\begin{proof}
``$\Leftarrow$'':  Suppose $H=(H,\bigtriangleup,\varepsilon,\Phi,S,\alpha,\beta)$.  By the definition of a radically graded quasi-Hopf algebra, $H[0]=(H[0], \bigtriangleup,\varepsilon,\Phi,S|_{H[0]},\alpha,\beta)$ is a quasi-Hopf subalgebra of $H$. If $H$ is not genuine, then there is a twist $J$ of $H$, such that $H^J$ is a Hopf algebra, i.e., $$\Phi_J=1\otimes 1\otimes 1,\ \ \ \alpha_J\beta_J=1.$$ Let $\pi:H\to H[0]=H/I$ be the natural projection, where $I=\oplus_{i\geq 1}H[i]$ is the Jacobson radical of $H.$ Define $J_0=(\pi\otimes \pi)(J)$.  It is clear that  $J=J_0+J_{\geq 1}$, where $J_{\geq 1}\in H\otimes I+I\otimes H.$
Since $\varepsilon(I)=0$,  we have $(id\otimes \varepsilon)(J_0)=(\pi\otimes \pi)(id\otimes \varepsilon)(J)=1.$ Similarly, $(\varepsilon\otimes id)(J_0)=1.$ It is obvious that $J_0$ has the inverse $(\pi\otimes \pi)(J^{-1})$ because $\pi$ is an algebra morphism. It follows that  $J_0$ is a twist for $H[0]$.
Now $\Phi\in H[0]^{\otimes 3}$  implies that
\begin{equation*}
\left\{
  \begin{array}{l}
     \Phi_{J_{0}}\in H[0]^{\otimes 3},\\
     \Phi_{J_{\geq 1}}\in H\otimes H\otimes I+H\otimes I\otimes H+I\otimes H\otimes H.
  \end{array}
\right.
\end{equation*}
Combining $\Phi_J=1\otimes 1\otimes 1,$ we obtain $\Phi_{J_0}=\Phi_J=1\otimes1\otimes 1.$ Similarly, we have $\alpha_{J_0}\beta_{J_0}=1.$ Hence, $H[0]^{J_0}$ is a Hopf algebra,  a contradiction to the fact that  $H[0]$ is genuine.

``$\Rightarrow$'':  If $H[0]$ is not genuine, then there is a twist $J$ of $H[0]$ such that $H[0]^{J}$ is a Hopf algebra. It is easy to see that $J$ is also a twist of $H$, and $H^J$ is a Hopf algebra.
\end{proof}

\begin{theorem}\label{T3.12}
Suppose that $u(\mathfrak{D},\lambda,\mu,\Phi_{\underline{c}})$ is a quasi-Hopf algebra of Cartan type with $\lambda=0, \mu=0$.  Then $u(\mathfrak{D},\lambda,\mu,\Phi_{\underline{c}})$ is a genuine quasi-Hopf algebra.
\end{theorem}
\begin{proof}Keep the same notations  as those in Theorem \ref{T3.4}.
Let $H=u(\mathfrak{D},0,0,\Phi_{\underline{c}})$ and the radically graded structure is given by $H=\oplus_{i\geq 0}H[i]$. Then $H[0]=\k G,$ and the associator $$\Phi_{\underline{c}}=\sum_{e,f,g\in G}\phi_{\underline{c}}(e,f,g)1_e\otimes 1_f\otimes 1_g$$ for a nonzero $\underline{c}\in \Gamma(\mathfrak{D}).$  By Proposition \ref{P2.16}, $\phi_{\underline{c}}$ is a $3$-cocycle on $G$, but not a coboundary. So $(H[0],\Phi_{\underline{c}})$ is a genuine quasi-Hopf algebra. It follows from Proposition \ref{P4.1} that $H=u(\mathfrak{D},0,0,\Phi_{\underline{c}})$ is a genuine quasi-Hopf algebra.
\end{proof}

\subsection{Some classification results} Let $H=\oplus_{i\geq 0}H[i]$ be a finite-dimensional radically graded quasi-Hopf algebra.  The ideal $I=\oplus_{i\geq 1}H[i]$ is the radical of $H$. Note that  $H[i]=I^i/I^{i+1}$.   In fact, we have the following relations:
\begin{equation}
H[i]=H[1]^i, i\geq 1.
\end{equation}
Assume that $H_0=\k[G]$ and $G$ an abelian group.  Then $(\k[G],\Phi)$ is a quasi-Hopf subalgebra of $H$ with the inherited associator $\Phi$ and the restricted  antipode $(S|_{H_0},\alpha,\beta)$. Now we  construct a new quasi-Hopf algebra $\widehat{H}$. By $\triangleright$ we denote the inner action of $H[0]$ on $H[1]$:
$$g\triangleright X=g\cdot X\cdot g^{-1},\ \ g\in G,\  X\in H[1],$$
where $\cdot$  stands for the multiplication of $H$. Extend the action of $H[0]$ on $H[1]$ to the tensor algebra $T(H[1])$ naturally.

Let $\widehat{H}$  be the smash product algebra $T(H[1])\rtimes H[0]$.  The algebra $\widehat{H}$ has a natural comultiplication given by
 $$\bigtriangleup_{\widehat{H}}(X)=\bigtriangleup_{H}(X),\bigtriangleup_{\widehat{H}}(g)=g\otimes g, X\in H[1],g\in G.$$ Let $S':\widehat{H}\to \widehat{H}\otimes \widehat{H}$ be a algebra antimorphism such that $S'|_{H[0]\otimes H[1]}=S|_{H[0]\otimes H[1]}.$   One may verify straightforward that $(\widehat{H},\bigtriangleup_{\widehat{H}},\Phi, S',\alpha,\beta)$  forms  a quasi-Hopf algebra. It is obvious that we have a canonical surjective homomorphism $P:\widehat{H}\to H$ such that $P$ restricts to the identity on $H[0]\oplus H[1].$

In what follows,  the elements in $H[n]$ will be said to be of degree $n$. In order to  classify the finite-dimensional radically graded quasi-Hopf algebras over abelian groups, we need the following proposition,  whose proof is parallel to \cite[Proposition 3.3.2, Proposition 3.3.3]{A}, hence will be omitted.

\begin{proposition}\label{P4.2}
Let $H$ be a finite-dimensional radically graded quasi-Hopf algebra, and $\pi:H\to u(\mathfrak{D},0,0,\Phi_{\underline{c}})$ a quasi-Hopf algebra epimorphism such that the restriction of  $\pi$ to the parts of  degree $0$ and $1$  is the identity. Then $ad_c(X_i)^{1-a_{ij}}(X_j)=0, i\neq j$ and $X_\alpha^{N_J}=0,\alpha\in R^+_J, J\in \Omega$ hold in $H.$
\end{proposition}
Now we are able to give one of the main results of the paper. The notations of $G, \mathbbm{G}, g_i,\mathbbm{g}_i,1\leq i\leq n$ are the same as those in Subsection 3.1.
\begin{theorem}\label{classify}
Suppose that $H$ is a radically graded finite-dimensional genuine quasi-Hopf algebra over an abelian group $G$ with an associator $\Phi=\sum_{e,f,g\in G}\phi(e,f,g)1_e\otimes 1_f\otimes 1_g,$ where  $\phi$ an abelian $3$-cocycle on $G$ satisfying $2,3,5,7\nmid \dim(H).$ Then $H\cong u(\mathfrak{D},0,0,\Phi_{\underline{c}})$ for some datum of finite Cartan type $\mathfrak{D}$ and some nonzero $\underline{c}\in \Gamma(\mathfrak{D}).$
\end{theorem}
\begin{proof}According to Subsection 2.5, we know that $\{\phi_{\underline{c}}|\underline{c}\in \mathcal{A},c_{r,s,t}=0,1\leq r<s<t\leq n\}$ is a complete set of representatives of abelian $3$-cocycles on $G$.  Thus there exists a $2$-cochain $J$ on $G$ such that $\phi\partial J=\phi_{\underline{c}}$ for some $\underline{c}$ in $\mathcal{A}$ satisfying  $c_{r,s,t}=0$ for all $1\leq r<s<t\leq n.$  Let $\mathbbm{J}=\sum_{f,g}J(f,g)1_f\otimes 1_g$.  It is clear that the associator of $H^{\mathbbm{J}}$ is $\Phi_{c}.$  Thus, without loss of generality, we may  assume that the associator of $H$ is $\Phi_{\underline{c}}$ for some $\underline{c}\in \{\underline{c}\in \mathcal{A}|c_{r,s,t}=0,1\leq r<s<t\leq n\}.$

 Since $H=\oplus_{i\geq 0}H[i]$ is a radically graded quasi-Hopf algebra, we can construct a new quasi-Hopf algebra $\widehat{H},$  so that there is an epimorphism
$P:\widehat{H}\to H$ such that $P$ restricts to the  identity   on $H[0]\oplus H[1]$. Denote by $L$ the sum of all quasi-Hopf ideals of $\widehat{H}$ contained in $\sum_{i\geq 2}H[i].$ It is easy to see that $Ker P\in L.$  Let $\overline{H}$ be the quotient $\widehat{H}/L$, and $\eta: \widehat{H}\to \overline{H}$ the canonical projection.  Thus,  there exists an epimorphism $\pi:H\to \overline{H}$ such that  $\eta=\pi P$.  Note that $\pi$ restricts to the identity as well on $H[0]\oplus H[1]$.

Next we  show that $\overline{H}\cong u(\mathfrak{D},0,0,\Phi_{\underline{c}})$ for some $\mathfrak{D}$ and some  $\underline{c}\in \Gamma(\mathfrak{D}).$ Then, it follows from Proposition \ref{P4.2} that $\pi$ must be an isomorphism, and the proof  will be done.
Decompose $H[1]=\oplus_{\chi\in \widehat{G}}H_\chi[1],$ where
\begin{equation}
H_\chi[1]=\{X\in H[1]|gXg^{-1}=\chi(g)X,\forall g\in G\}.
\end{equation}
For each $\chi\in \widehat{G},$ define a $\widetilde{\chi}\in \widehat{\mathbbm{G}}$ by
$\widetilde{\chi}(\mathbbm{g}_i)=\chi(g_i)^{\frac{1}{m_i}}, 1\leq i\leq n.$ Denote by $\widetilde{H}$ the quasi-Hopf algebra generated by $\overline{H}$ and $\mathbbm{g}_i,1\leq i\leq n,$ where $\mathbbm{g}_i^{m_i}=g_i,$ and $\mathbbm{g}X\mathbbm{g}^{-1}=\widetilde{\chi}(\mathbbm{g})X$ for all $\mathbbm{g}\in \mathbbm{G}$ and $X\in H_\chi[1].$ It is obvious that $\widetilde{H}$ is a radically graded quasi-Hopf algebra over $\mathbbm{G},$ and $\overline{H}$ is the quasi-Hopf subalgebra of $\widetilde{H}$ generated by $\widetilde{H}[1]$ and $\mathbbm{g}_i,1\leq i\leq n.$

Consider $A=\widetilde{H}^{\mathbbm{J}_{\underline{c}}^{-1}}$.  Since $A$ is a finite-dimensional radically  graded Hopf algebra over $\mathbbm{G},$ and $A$ is of the form $R\#\k \mathbbm{G}$ for some braided graded Hopf algebra in the category of Yeter-Drinfeld modules over $\mathbbm{G}.$  So $A$ is also a finite-dimensional pointed Hopf algebra over $\mathbbm{G}.$ By the classification result of \cite{AS2}, there exists a datum of finite Cartan type $\mathfrak{D}=\mathfrak{D}(\mathbbm{G},(h_i)_{1\leq i\leq \theta},(\chi_i)_{1\leq i\leq \theta}$ such that $A=u(\mathfrak{D},0,0)$. Since $\overline{H}$ is generated by $A^{{\mathbbm{J}}_{\underline{c}}}[1]$ and $\mathbbm{g}_i,1\leq i\leq n,$  there is a nonzero $\underline{c}\in \Gamma(\frak{D})$ such that $\overline{H}\cong u(\mathfrak{D},0,0,\Phi_{\underline{c}})$ by Proposition \ref{P3.10}.
\end{proof}

From Proposition \ref{P2.16} we know that every $3$-cocycle on a cyclic group or on an abelian group of the form $\Z_{m_1}\times \Z_{m_2}$  is abelian.  So we have the following.

\begin{corollary}
Suppose that $H$ is a radically graded finite-dimensional genuine quasi-Hopf algebra over an abelian group $G=\Z_m\otimes \Z_n$ such that $2,3,5,7\nmid \dim(H).$ Then $H\cong u(\mathfrak{D},0,0,\Phi_{\underline{c}})$ for some datum of finite Cartan type $\mathfrak{D}$ and some $\underline{c}\in \Gamma(\mathfrak{D}).$
\end{corollary}

\section{Small quasi-quantum groups}
In this section, we will introduce small quasi-quantum groups. These algebras can be viewed as  natural generalization of small quantum groups. We will  present several examples of small quasi-quantum groups which are genuine quasi-Hopf algebras. We fix a  finite abelian group $G$ with free generators $\{g_i|1\leq i\leq n\}$.  The notations $\mathbbm{G}$ and $\{\mathbbm{g}_i|1\leq i\leq n\}$ are defined in the the same way as those  in  Subsection 3.1.

\subsection{Small quasi-quantum groups} Suppose that $A=(a_{ij})_{1\leq i,j\leq n}$ is a finite Cartan matrix and that  $\mathbf{D}(A)$ is the Cartan matrix
$\left(
\begin{array}{cc}
A & 0 \\
0 & A \\ \end{array}
 \right)$.
 Let $G=\Z_m^n=\langle g_1\rangle\times\cdots \times  \langle g_n\rangle.$

\begin{definition}
Let $\mathfrak{D}(\mathbbm{G},(h_i)_{1\leq i\leq 2n}, (\chi_i)_{1\leq i\leq 2n}, \mathbf{D}(A))$ be a datum of finite Cartan type such that $h_i=h_{n+i}\in G, \chi_i=\chi^{-1}_{i+n},\ 1\leq i\leq n.$ Suppose that $\underline{c}\in \Gamma(\mathfrak{D})$ is nonzero. We call
the quasi-Hopf algebra $Qu(\mathfrak{D},\lambda,\Phi_{\underline{c}})=u(\mathfrak{D},\lambda,\mu,\Phi_{\underline{c}})$  a quasi-small quantum group if $\mu=0$ and $\lambda_{i,j}\neq 0$ if and only if $j=i+n$ for $1\leq i\leq n.$
\end{definition}

The quasi-Hopf algebras in Example \ref{E3.14} are examples of quasi-small quantum groups. More examples will be given before we show that small quasi-quantum groups are natural generalization of small quantum groups. For $1\leq i\leq n,$ define $E_i=X_i$, $F_i=X_{i+n}h_i^{-1}$ and $X'_i=X_{i+n}.$ Let $V=\k\{E_1,\cdots,E_n\}$, $V'=\k\{X'_1,\cdots,X'_n\}$ and $U=\k\{F_1,\cdots,F_n\}.$ It is obvious that $V$ and $V'$ are braided vector spaces of Cartan type with the braiding matrices $(q_{ij})_{1\leq i,j\leq n}$ and $(q^{-1}_{ij})_{1\leq i,j\leq n}$ respectively, where $q_{ij}=\chi_j(h_i), 1\leq i,j\leq n.$  Note that By definition de braided vector spaces of Cartan type (see Subsection 2.3), the associated Cartan matrices of $V$ and $V'$  are both equal to $A.$  Let $q'_{ij}=q^{-1}_{ji},1\leq i,j\leq n$. We define a braiding $c$ on $U$ as follows:
$$c(F_i\otimes F_j)= q'_{ij}F_j\otimes F_i, 1\leq i,j\leq n.$$
For all $1\leq i,j\leq n,$ we have $q'_{ij}q'_{ji}=q^{-1}_{ji}q^{-1}_{ij}=q_{ii}^{-a_{ij}}={q'}_{ii}^{a_{ij}}.$  So $(U,c)$ is also a braided vector space of Cartan type, and the associated Cartan matrix is $A$  as well.  So we can define braided commutators, the braided adjoint action and the root vectors over $T(U)$, see Subsection 2.3 for details.
Now let $R^+$ be the positive root system corresponding to the Cartan matrix  $A$ with respect to the simple roots $\alpha_1,\cdots, \alpha_n,$  and $F_\alpha, \alpha\in R^+$ the root vectors such that $F_{\alpha_i}=F_i$ for $1\leq i\leq n.$  Let $\Omega$ be the set of the connected components of $I=\{1,\cdots,n\},$  and  $R^+_J$ be the positive root system corresponding to $J\in \Omega.$

\begin{proposition}\label{P5.2} $Qu(\mathfrak{D},\lambda,\Phi_{\underline{c}})$ is generated by $g_i, E_i,F_i, 1\leq i\leq n$ subject to the relations:
\begin{eqnarray}
&&g_iE_jg_i^{-1}=\chi_j(g_i)E_j,\ g_iF_jg_i^{-1}=\chi^{-1}_j(g_i)F_j,\ 1\leq i,j\leq n, \label{5.1}\\
&&E_iF_j-F_jE_i=\delta_{ij}\lambda_{i,i+n}(h_i^{-1}-h_i),\ \lambda_{i,i+n}\neq 0,\ 1\leq i,j\leq n,\label{5.2}\\
&&ad_c(E_i)^{1-a_{ij}}(E_j)=0,\ ad_c(F_i)^{1-a_{ij}}(F_j)=0,\ 1\leq i\neq j\leq n,\label{5.3}\\
&&E_\alpha^{N_J}=0,\ F_\alpha^{N_J}=0,\ \alpha\in R_J^+,\ J\in\Omega.\label{5.4}
\end{eqnarray}
The comultiplication is determined by
\begin{eqnarray}
&&\bigtriangleup(E_i)=\sum_{f,g\in G}\Psi_i(f,g)E_i1_f\otimes 1_g+h_i\otimes E_i,\label{5.5}\\
&&\bigtriangleup(F_i)=\sum_{f,g\in G}\Psi_{i+n}(f,g)\chi_g(h_i^{-1})F_i1_f\otimes 1_g+1\otimes F_i.\label{5.6}
\end{eqnarray}
The associator is $\Phi_{\underline{c}}$
and the antipode $(\mathcal{S},\mathcal{\alpha},1)$ is given by
\begin{eqnarray}
&\mathcal{\alpha}=\sum_{g\in G}\Upsilon(g)1_g,&\label{5.7}\\
&\mathcal{S}(F_i)=\chi^{-1}_{i}(h_i)\sum_{g\in G}\chi_g(h_i)\mathcal{F}_{i+n}(g)F_i1_g,\ \mathcal{S}(E_i)=\sum_{g\in G}\mathcal{F}_i(g)E_i1_g,&\label{5.8}
\end{eqnarray} for $1\leq i\leq n.$
\end{proposition}
\begin{proof}
By Theorem \ref{T3.4}, $Qu(\mathfrak{D},\lambda,\Phi_{\underline{c}})$ is generated by $G$ and $X_j,1\leq i\leq n, 1\leq j\leq 2n$ subject to the relations \eqref{e3.20}-\eqref{e3.23}. Since $F_ih_i=X_i,$ the algebra  $Qu(\mathfrak{D},\lambda,\Phi_{\underline{c}})$ is also generated by  $g_i, E_i,F_i, 1\leq i\leq n.$ Now  we  show that the relations \eqref{5.1}-\eqref{5.4} are equivalent to the relations \eqref{e3.20}-\eqref{e3.23}. It is easy to see that \eqref{5.1} equals \eqref{e3.20}.
For all $1\leq i,j\leq n,$ we have:
\begin{eqnarray*}
E_iF_j-F_jE_i&=&X_iX_{j+n}h_j^{-1}-X_{j+n}h_j^{-1}X_i\\
&=&X_iX_{j+n}h_j^{-1}-\chi_{j+n}(h_j^{-1})X_iX_{j+n}h_j^{-1}\\
&=&X_iX_{j+n}h_j^{-1}-\chi_{j}(h_j)X_iX_{j+n}h_j^{-1}\\
&=&ad_c(X_i)(X_j)h_j^{-1}.
\end{eqnarray*}
So Relation \eqref{5.2} is equivalent  to  Relation \eqref{e3.22}. For $1\leq i\neq j\leq n$, we have
\begin{eqnarray*}
ad_c(F_i)(F_j)&=&F_iF_j-{q'}_{ij}F_jF_i=F_iF_j-\chi_i^{-1}(h_j)F_jF_i\\
&=&X_{i+n}h_i^{-1}X_{j+n}h_j^{-1}-\chi_i^{-1}(h_j)X_{j+n}h_j^{-1}X_{i+n}h_i^{-1}\\
&=&\chi^{-1}_j(h_i^{-1})X_{i+n}X_{j+n}h_i^{-1}h_j^{-1}-\chi_i^{-1}(h_j)\chi^{-1}_i(h_j^{-1})X_{j+n}X_{i+n}h_i^{-1}h_j^{-1}\\
&=&\chi_j(h^i)[X_{i+n}X_{j+n}-\chi^{-1}_j(h^i)X_{j+n}X_{i+n}]h_i^{-1}h_j^{-1}\\
&=&\chi_j(h^i)[ad_c(X_{i+n})(X_{j+n})]h_i^{-1}h_j^{-1}.
\end{eqnarray*}
By induction on $l\geq 1,$  one can show that
\begin{equation}
(ad_c(F_i))^{l}(F_j)=\chi_j^l(h_i)\chi_i^{\frac{l(l-1)}{2}}(h_i)[(ad_c(X_{i+n}))^l(X_{j+n})]h_i^{-l}h_j^{-1}.
\end{equation}
Hence, the relation \eqref{5.3} is equivalent to the relation \eqref{e3.21}. Similarly, for each $\alpha\in R^+,$  one can show that
$$F_\alpha=\lambda_\alpha X'_{\alpha} h_\alpha^{-1},$$
where $\lambda_\alpha$ is some nonzero number depend on $\alpha,$ and $h_\alpha$ is defined by \eqref{e2.7}.  Thus,  for each $\alpha\in R^+_J, J\in \Omega,$  we have
$F_\alpha^{N_J}=0$ if and only if ${X'}_{\alpha}^{N_J}=0$. Therefore, the relation \eqref{5.4} equivalent to  the relation \eqref{e3.23} since $\mu=0$.
We have proved that $Qu(\mathfrak{D},\lambda,\Phi_{\underline{c}})$ is generated by $g_i, E_i,F_i, 1\leq i\leq n$ subject to the  relations \eqref{5.1}-\eqref{5.4}.

Next we compute the comultiplication and the antipode of $Qu(\mathfrak{D},\lambda,\Phi_{\underline{c}})$ for the generators.
Formula \eqref{5.5} follows from the fact $h_i=\sum_{f\in G}\Theta_i(f)1_f$ for $h_i\in G$.  Formula \eqref{5.6}  holds because of the following equations:
\begin{eqnarray*}
\bigtriangleup(F_i)&=&\bigtriangleup(X_{i+n}h_i^{-1})\\
&=&(\sum_{f,g\in G}\Psi_{i+n}(f,g)X_{i+n}1_f\otimes 1_g+\sum_{f\in G}\Theta_{i+n}(f)1_f\otimes X_{i+n})(h_i^{-1}\otimes h_i^{-1})\\
&=&\sum_{f,g\in G}\Psi_{i+n}(f,g)\chi_g(h_i^{-1})F_i1_f\otimes 1_g+\sum_{f\in G}\Theta_{i+n}(f)\chi_f(h_i^{-1})1_f\otimes F_i\\
&=&\sum_{f,g\in G}\Psi_{i+n}(f,g)\chi_g(h_i^{-1})F_i1_f\otimes 1_g+1\otimes F_i.
\end{eqnarray*}
Formula \eqref{5.7} is obvious.  Formula \eqref{5.8} follows from
\begin{eqnarray*}
\mathcal{S}(F_i)&=&\mathcal{S}(h_i^{-1})\mathcal{S}(X_{i+n})=h_i\sum_{g\in G}\mathcal{F}_{i+n}(g)X_{i+n}1_g\\
&=&\chi_{i+n}(h_i)\sum_{g\in G}\chi_g(h_i^{2})\mathcal{F}_{i+n}(g)F_i1_g,
\end{eqnarray*}
where $\chi_g, g\in G$,  is defined by \eqref{3.1}.
\end{proof}

Now let $\mathfrak{D}(\mathbbm{G},(h_i)_{1\leq i\leq 2n}, (\chi_i)_{1\leq i\leq 2n}, \mathbf{D}(A))$ be a datum of finite Cartan type such that $h_i=h_{n+i}\in G, \chi_i=\chi^{-1}_{i+n},\ 1\leq i\leq n$.  Note that $\Gamma(\mathfrak{D})$ is not empty since $0\in \Gamma(\mathfrak{D}).$

 Take an element $\underline{c}\in \Gamma(\mathfrak{D})$.  We define an algebra $H_{\underline{c}}$ generated by $G$ and $E_i,F_i,1\leq i\leq n$ subject to the Relations \eqref{5.1}-\eqref{5.4}. Define an algebra morphism $\bigtriangleup:H_{\underline{c}}\to H_{\underline{c}}\otimes H_{\underline{c}}$ by \eqref{5.5}-\eqref{5.6}, and an algebra antimorphism $S:
H_{\underline{c}}\to H_{\underline{c}}$ by \eqref{5.8}. Let $\alpha$ be an element of $H_{\underline{c}}$ defined by \eqref{5.7}. Define an algebra morphism $\varepsilon:H_{\underline{c}}\to \k$ such that $\varepsilon(g)=1, \varepsilon(E_i)=\varepsilon(F_i)=0$ for $g\in G, 1\leq i\leq n.$   We have the following identification of the algebra $H_{\underline{c}}$.

\begin{proposition}\label{P5.3}
\begin{itemize}
\item[(1)]  If $\underline{c}=0$, then $(H_{\underline{c}},\bigtriangleup,\varepsilon,S)$ is isomorphic to the AS-Hopf algebra $u(\mathfrak{D}',\lambda,0),$ where $\mathfrak{D}'= \mathfrak{D}(G,(h_i)_{1\leq i\leq 2n}, (\chi'_i)_{1\leq i\leq 2n}, \mathbf{D}(A)),$ and $\chi'_i=\chi_i|_G, 1\leq i\leq n.$
    Each small quantum group is isomorphic to a $H_{\underline{0}}$ as Hopf algebras.
\item[(2)] If $\underline{c}\neq 0,$ then $(H_{\underline{c}},\bigtriangleup,\varepsilon,\Phi_{\underline{c}},S,\alpha,1)$ is  isomorphic to the small quasi-quantum group $Qu(\mathfrak{D},\lambda,\Phi_{\underline{c}}).$
\end{itemize}
\end{proposition}
\begin{proof}
Observe that  the functions defined by \eqref{e3.14}-\eqref{e3.16} are equivalent to the constant function $1$ if $\underline{c}=0$.  Hence, the algebra morphism $\Upsilon:H_{\underline{c}}\to u(\mathfrak{D}',\lambda,0)$ given by $$\Upsilon(g)=g,\Upsilon(E_i)=X_i, \Upsilon(F_i)=X_{i+n}h_i,$$ for all
$g\in G, 1\leq i\leq n,$ is a Hopf algebra isomorphism. The second part of the proposition follows from Proposition \ref{P5.2}.
\end{proof}
From this proposition, we can see that small quasi-quantum groups are natural generalizations of small quantum groups.

\subsection{Triangular decomposition and half small quasi-quantum groups} In this subsection, we study the triangular decomposition of small quasi-quantum groups.  Fix an abelian group $G=\Z_m^n=\langle g_1\rangle\times\cdots \times  \langle g_n\rangle$ and a finite Cartan matrix $A=(a_{ij})_{1\leq i,j\leq n}.$  Assume that $\mathfrak{D}=\mathfrak{D}(\mathbbm{G},(h_i)_{1\leq i\leq 2n}, (\chi_i)_{1\leq i\leq 2n}, \mathbf{D}(A))$ is a datum of finite Cartan type and $Qu(\mathfrak{D},\lambda,\Phi_{\underline{c}})$ is a small quantum group.
Denote by $Qu^+(\mathfrak{D},\Phi_{\underline{c}})$ (resp. $Qu^-(\mathfrak{D},\Phi_{\underline{c}})$) the subalgebra generated by $E_i,1\leq i\leq n$ (resp. $F_i,1\leq i\leq n$), and $Qu^0=\k G.$ So we have a natural linear isomorphism \begin{eqnarray*}
\varphi:Qu^+(\mathfrak{D},\Phi_{\underline{c}})\otimes Qu^0\otimes Qu^-(\mathfrak{D},\Phi_{\underline{c}})&\To& Qu(\mathfrak{D},\lambda,\Phi_{\underline{c}})\\
          x\otimes y\otimes z&\to& x yz,
\end{eqnarray*}
for all $x\in Qu^+(\mathfrak{D},\Phi_{\underline{c}}),\ y\in Qu^0,\ z\in Qu^-(\mathfrak{D},\Phi_{\underline{c}}).$  This decomposition $Qu(\mathfrak{D},\lambda,\Phi_{\underline{c}})=Qu^+(\mathfrak{D},\Phi_{\underline{c}})\otimes Qu^0\otimes Qu^-(\mathfrak{D},\Phi_{\underline{c}})$ is called   {\bf{the triangular decomposition}} of  $Qu(\mathfrak{D}, \lambda,\Phi_{\underline{c}})$.  Denote by
$Qu^{\geq 0}(\mathfrak{D},\Phi_{\underline{c}})$ (resp. $Qu^{\leq 0}(\mathfrak{D},\Phi_{\underline{c}})$) the subalgebra of $Qu(\mathfrak{D},\lambda,\Phi_{\underline{c}})$ generated by $G$ and $E_i,1\leq i\leq n$ (resp. $G$ and $E_i,1\leq i\leq n$). The following isomorphisms are obvious.
\begin{eqnarray}
Qu^{\geq 0}(\mathfrak{D},\Phi_{\underline{c}})\cong u(\mathfrak{D}',0,0,\Phi_{\underline{c}}),\\
Qu^{\leq 0}(\mathfrak{D},\Phi_{\underline{c}})\cong u(\mathfrak{D}'',0,0,\Phi_{\underline{c}}),
\end{eqnarray}
where $\mathfrak{D}'=\mathfrak{D}(\mathbbm{G},(h_i)_{1\leq i\leq n},(\chi_i)_{1\leq i\leq n},A),$
$\mathfrak{D}''=\mathfrak{D}(\mathbbm{G},(h_i)_{n+1\leq i\leq 2n},(\chi_i)_{n+1\leq i\leq n},A).$ The two radically graded quasi-Hopf algebras are called
{\bf{the half small quasi-quantum groups}}.
Keep the assumption of $G$ and $A$ as above, we have the following.
\begin{proposition}
Let $\mathfrak{D}=\mathfrak{D}(\mathbbm{G},(h_i)_{1\leq i\leq n},(\chi_i)_{1\leq i\leq n},A)$ be a datum of finite Cartan type such that $h_i\in G,1\leq i\leq n,$ and $\chi_i(h_j)=\chi_j(h_i), 1\leq i,j\leq n.$ Then for any nonzero $\underline{c}\in \Gamma(\mathfrak{D}),$ $u(\mathfrak{D},0,0,\Phi_{\underline{c}})$ is a half small quasi-quantum group.
\end{proposition}
\begin{proof}
We need to show that there exists a small quasi-quantum group $Qu(\mathfrak{D}',\lambda, \Phi_c)$ such that $u(\mathfrak{D},0,0,\Phi_{\underline{c}})\cong Qu^{\geq 0}(\mathfrak{D}', \Phi_{\underline{c}}).$ Define $h_{i+n}=h_i$ and $\chi_{i+n}=\chi_i^{-1}$ for $1\leq i\leq n.$ For  $1\leq i,j\leq n,$ we have the following:
$$\chi_i(h_{j+n})\chi_{j+n}(h_i)=\chi_i(h_j)\chi_j^{-1}(h_i)=1.$$
It is clear that
$\mathfrak{D}'=\mathfrak{D}(\mathbbm{G},(h_i)_{1\leq i\leq 2n},(\chi_i)_{1\leq i\leq 2n},D(A))$ is a datum of finite Cartan type. For each $1\leq i\leq n$, let $h_i=\prod_{j=1}^n\mathbbm{g}^{s_{ij}}$.  Since $h_i\in G$, we have  $s_{ij}\equiv 0 \mod m.$  It follows that  $\Gamma(\mathfrak{D})=\Gamma(\mathfrak{D}').$
Moreover,  $h_ih_{i+j}=h_i^2\in G$ and $\chi_i\chi_{i+n}=\chi_i\chi_i^{-1}=\varepsilon$. Thus,  we obtain a family of linking parameters $\lambda=(\lambda_{ij})_{1\leq i<j\leq 2n, i\nsim j}$ for $\mathfrak{D}'$ such that $\lambda_{i,j}\neq 0$ if and only if $j=i+n.$  Since each connected component of $I=\{1,\cdots,n\}$ satisfies \eqref{2.7},  we have  $h_ih_{i+n}=h_i^2\neq 1.$ This implies that $\lambda$ is a family of modified linking parameters. Therefore,
we  obtain a small quasi-quantum group $Qu(\mathfrak{D}',\lambda,\Phi_{\underline{c}})$  such that $u(\mathfrak{D},0,0,\Phi_{\underline{c}})\cong Qu^{\geq 0}(\mathfrak{D}', \Phi_{\underline{c}}).$
\end{proof}

\subsection{Examples of genuine small quasi-quantum groups} Let $A=(a_{ij})_{1\leq i,j\leq n}$ be a finite Cartan matrix, $(d_1,\cdots,d_n)$ a vector  with  elements in $\{1,2,3\}$ such that $(d_ia_{ij})_{1\leq i,j\leq n}$ is symmetric. Let $q$ be an $N$-th primitive root of unity, where $N$ is an odd positive integer. Moreover, in case $A$ has a connected component of type $G_2,$ we will add  one more assumption that $N$ is prime to $3.$ Let $p$ be a positive odd integer and $m=pN.$ Choose $\zeta_{m^2}$ and $\zeta_m$ such that $\zeta_{m^2}^m=\zeta_m$, $\zeta_m^p=q.$ Let $G=Z_m^n=\langle g_1\rangle\times \cdots\times \langle g_n\rangle$,  and $l$, $1\leq l<m$,  be an integer such that $lp \neq 0 \mod m$. Define characters $\{\chi_i| 1\leq i\leq n\}$ on $\mathbbm{G}$ as follows:
\begin{equation}
\chi_i(\mathbbm{g}_j)=\left\{
                        \begin{array}{ll}
                          \zeta_{m^2}^{2pd_i}, & \hbox{if $i=j$;} \\
                          \zeta_{m^2}^{pd_ia_{ij}}, & \hbox{if $i\neq j$ and $a_{ij}\neq 0$;} \\
                           1, & \hbox{if $a_{ij}=0$.}
                        \end{array}
                      \right.
\end{equation}
With the above notations, it is easy to verify that $\mathfrak{D}=\mathfrak{D}(\mathbbm{G},(h_i)_{1\leq i\leq 2n},(\chi_i)_{1\leq i\leq 2n},\mathbf{D}(A))$ is a datum of finite Cartan type, where
$h_i=h_{i+n}=\mathbbm{g}_i^{ml}$ and $\chi_{i+n}=\chi_i^{-1}$ for $1\leq i\leq n.$
\begin{lemma}\label{L4.3}
There are nonzero elements $\underline{c}=(c_i,c_{jk})_{1\leq i\leq n,1\leq j<k\leq n}$ in $\Gamma(\mathfrak{D})$ such that $c_i=k_iN $ for  some $1\leq k_i<p,\ 1\leq i\leq n.$
\end{lemma}
\begin{proof}
It is obvious that \eqref{3.8}-\eqref{3.9} are solvable for $(c_{jk})_{1\leq j<k\leq n}.$ Thus, it suffuces to prove that Equation \eqref{3.7} has solutions $c_i=k_iN $ for some $1\leq k_i<p,\ 1\leq i\leq n.$  Now the equations in variable $c_i$ are
\begin{eqnarray}
&&ml\equiv 2c_ipd_i \mod m,\label{4.5}\\
&&ml\equiv -2c_ipd_i \mod m,\label{4.6}\\
&&0\equiv c_ipd_ia_{ij} \mod m,\ i\neq j,a_{ij}\neq 0,\label{4.7}\\
&&0\equiv -c_ipd_ia_{ij} \mod m,\ i\neq j,a_{ij}\neq 0.\label{4.8}
\end{eqnarray}
It is not difficult to see that Equations \eqref{4.5}-\eqref{4.8} have  a set of solutions
\begin{equation}\label{4.9}
\{c_i=kN|0\leq k<p\}.
\end{equation} We have thus  proved the lemma.
\end{proof}
We need the following lemma.
\begin{lemma}\label{L4.4}
There exists a family of modified linking parameters $\lambda$ for $\mathfrak{D}$ satisfying condition:
$\lambda_{ij}\neq 0$ if and only if $i+n=j.$
\end{lemma}
\begin{proof}
Note  that for each $1\leq i\leq n,$ we have  $\chi_i\chi_{i+n}=\varepsilon$ because $\chi_{i+n}=\chi_i^{-1}.$ The fact that  $lp\neq 0 \mod m,$  implies that $l\neq 0 \mod N$. Hence $2l\neq 0 \mod N$ because $N$ is odd. It follows that $h_ih_{i+n}=\mathbbm{g}_i^{2ml}=g_i^{2l}\in G$ and $h_ih_{i+n}\neq 1.$  By definition, there exists a family of modified linking parameters $\lambda$ for $\mathfrak{D}$ such that
$\lambda_{ij}\neq 0$ if and only if $i+n=j.$
\end{proof}
From Lemmas \ref{L4.3}-\ref{L4.4} and the definition of small quasi-quantum groups, we obtain the following.
\begin{proposition}\label{P4.5}
Let $\underline{c}$ be a nonzero parameter in $\Gamma(\mathfrak{D}),$ and $\lambda$  a a family of modified linking parameters such that $\lambda_{ij}\neq 0$ if and only if $i+n=j$. Then $u(\mathfrak{D},\lambda,0,\Phi_{\underline{c}})$ is a small quasi-quantum group.
\end{proposition}

 In what follows,  we let $Qu(\mathfrak{D},\lambda,\Phi_{\underline{c}})=u(\mathfrak{D},\lambda,0,\Phi_{\underline{c}})$, $\underline{c}\not=0$,  be a small quasi-quantum group given in Proposition \ref{P4.5}. We conjecture that  $Qu(\mathfrak{D},\lambda,\Phi_{\underline{c}})$ is  genuine.  At this moment, we are not able to prove it in general. But  we  have  the following partial result.

\begin{proposition}\label{P4.6}
Suppose $l>1$, $l|m$ and $l\nmid c_i$ for $1\leq i\leq n$.  Then $Qu(\mathfrak{D},\lambda,\Phi_{\underline{c}})$ is genuine.
\end{proposition}
\begin{proof}
Let $I$ be the quasi-Hopf ideal of $Qu(\mathfrak{D},\lambda,\Phi_{\underline{c}})$ generated by $\{X_i|1\leq i\leq n\}.$ Set $\widetilde{u}=Qu(\mathfrak{D},\lambda,\Phi_{\underline{c}})/I$.  It is evident that $\widetilde{u}\cong \k G'$, where $G'=G/\langle g_i^{2l}-1|1\leq i\leq n\rangle.$ Since gcd$(2l,m)=l,$  we have $G'=G/\langle g_i^{l}-1|1\leq i\leq n\rangle.$ For an element $g\in G,$  we denote by $\overline{g}$  the corresponding element in the quotient group $G'$. Let $G^\circ$ be the subgroup of $G$ generated by $\{ g_i^{\frac{m}{l}}|1\leq i\leq n\}.$ Define a group isomorphism $\varphi:G^\circ\to G'$ by $\varphi(g_i^{\frac{m}{l}})=\overline{g_i}, 1\leq i\leq n.$

Let $f=\prod_{1}^{n}g_i^{n_i}$.  We have the following expression of the element $\overline{1_f}$ in $\k \overline{G}$:
\begin{eqnarray*}
\overline{1_f}&=&\frac{1}{|G|}\sum_{g\in G}\chi_f(g)\overline{g}\\
&=&\frac{1}{|G|}\sum_{0\leq i_j<l,0\leq k_j<\frac{m}{l},1\leq j\leq n}[\prod_j^n\zeta_{m}^{n_j(i_j+k_jl)}(\prod_j^n\overline{g_j}^{i_j})]\\
&=&\frac{1}{|G|}\prod_j^{n}[\sum_{0\leq i_j<l,0\leq k_j<\frac{m}{l}}\zeta_{m}^{n_j(i_j+k_jl)}(\overline{g_j}^{i_j})]
\end{eqnarray*}
It follows that $\overline{1_f}\neq 0$ if and only if $f\in G^\circ.$
Now we assume that  $f=\prod_jg_j^{a_j\frac{m}{l}}\in G^\circ$.  Then we have:
\begin{eqnarray}\label{e5.18}
\overline{1_f}&=&\frac{1}{|G|}\prod_j^{n}[\sum_{0\leq i_j<l,0\leq k_j<\frac{m}{l}}\zeta_{m}^{a_j\frac{m}{l}(i_j+k_jl)}(\overline{g_j}^{i_j})]\\
&=&\frac{m}{l|G|}\prod_j^{n}[\sum_{0\leq i_j<l}\zeta_{m}^{a_j\frac{m}{l}(i_j)}(\overline{g_j}^{i_j})] \notag \\
&=&\frac{1}{|G'|}[\sum_{0\leq i_j<l,1\leq j\leq n}[\prod_j^{n}\zeta_{l}^{a_j(i_j)}\prod_j^{n}(\overline{g_j}^{i_j})]\notag \\
&=&1_{\varphi(f)}.\notag
\end{eqnarray}
Define a $3$-cocycle $\phi'_{\underline{c}}$ on $G'$ as follows:
\begin{equation}
\phi'_{\underline{c}}(e,f,g)=\prod_{1\leq i\leq n}\zeta_{l}^{c_it_i[\frac{r_i+s_i}{l}]}
\prod_{1\leq j<k\leq n}\zeta_{l}^{c_{jk}t_k[\frac{r_j+s_j}{l}]}
\end{equation}
 for $e=\prod_{i=1}^n \overline{g_i}^{r_i}, \  f=\prod_{i=1}^n \overline{g_i}^{s_i}, \  g=\prod_{i=1}^n \overline{g_i}^{t_i}.$

Now we compute the associator $\overline{\Phi_{\underline{c}}}$ of $\k G'$:
\begin{eqnarray*}
\overline{\Phi_{\underline{c}}}&=&\sum_{e,f,g\in G}\phi_{\underline{c}}(e,f,g)\overline{1_e}\otimes \overline{1_f}\otimes \overline{1_g}\\
&=&\sum_{e,f,g\in G^\circ}\phi_{\underline{c}}(e,f,g)1_{\varphi(e)}\otimes 1_{\varphi(f)}\otimes 1_{\varphi(g)}\\
&=&\sum_{e,f,g\in G'}\phi'_{\underline{c}}(e,f,g)1_{e}\otimes 1_{f}\otimes 1_{g}.
\end{eqnarray*}
The second identity follows from \eqref{e5.18}, and $\overline{1_f}=0$ if $f\notin G^\circ.$ The third identity follows from the definition of $\varphi$.
Since $l\nmid c_i$ for $1\leq i\leq \theta,$  we know that $\phi'_{\overline{c}}$ is a $3$-cocycle on $G'$ according to Proposition \ref{P2.16}. Moreover,  $\phi'_{\overline{c}}$ is not a $3$-coboundary. Hence, $\widetilde{u}\cong (\k G',\overline{\Phi_{\underline{c}}})$ is a genuine quasi-Hopf algebra.

Suppose that $Qu(\mathfrak{D},\lambda,\Phi_{\underline{c}})$ is not genuine. Then $Qu(\mathfrak{D},\lambda,\Phi_{\underline{c}})$ is twist equivalent to a Hopf algebra, or equivalently the tensor category $\mathcal{M}$ of representations of $Qu(\mathfrak{D},\lambda,\Phi_{\underline{c}})$ has a fiber functor. Let $Rep_{\Phi'_{\overline{c}}}(\k G')$ be the representation category of the quasi-Hopf algebra $\widetilde{u}\cong \k G'.$ It is obvious that $Rep_{\Phi'_{\overline{c}}}(\k G')$ is a full tensor subcategory of $\mathcal{M}$. Hence $Rep_{\phi'_{\overline{c}}}(\k G')$ has a fiber functor as well.  Thus $\widetilde{u}$ is twist equivalent to a Hopf algebra,  a contradiction since $\widetilde{u}$ is genuine.
\end{proof}
It is easy to see that there are (infinitely)  many numbers $l,p,N$ to choose such that $l,m,c_i,\ 1\leq i\leq n$ satisfy the conditions in  Proposition \ref{P4.6}. Therefore,  we obtain many examples of genuine small quasi-quantum groups associated to each finite Cartan matrix.

\section{Nonradically graded genuine Quasi-Hopf algebras associated to connected Cartan matrices}
In this section, we construct many new examples of nonradically graded genuine quasi-Hopf algebras associated to each finite connected Cartan matrix. Explicitly, for each finite connected Cartan matrix $A$, we will show that there exists a Cartan datum $\mathfrak{D}$ associated to $A$, nontrivial modified root vector parameters $\mu$ for $\mathfrak{D}$ and a nonzero $\underline{c}\in \Gamma({\mathfrak{D}})$ such that $u(\mathfrak{D},0,\mu,\Phi_{\underline{c}})$ is a finite-dimensional genuine quasi-Hopf algebras. All these quasi-Hopf algebras have only trivial linking relations since the Cartan matrices are assumed to be connected. In the following, the groups $\mathbbm{G}=\langle \mathbbm{g}_1\rangle \times \cdots \times \langle \mathbbm{g}_n\rangle$ and $G=\langle g_1\rangle\times \cdots \times \langle g_n\rangle$ are defines the same as those in Subsection 3.1. For each fixed Cartan matrix $A$, the corresponding root system $R,$ positive root system $R^+$ with simple roots $\{\alpha_i|1\leq i\leq n\}$ are defined to be the same as those in Subsection 2.2. Throughout this section, $p,q,d$ are positive odd numbers satisfying $(p,q)=(p,d)=(q,d)=1$.

\subsection{Quasi-Hopf algebras of Cartan type $A_n$, $B_n$ and $C_n$} Notice that examples of non-radically graded quasi-Hopf algebras associated to $A_2,B_2 (=C_2),G_2$ have been given in Subsection 3.5, so in this subsection we always assume $n\geq 3$.
Let $A=(a_{ij})_{1\leq i,j\leq n}$ be the Cartan matrix of type $A_n, B_n$ or $C_n.$  Let
 $G=\langle g_1\rangle \times \langle g_2\rangle \times \cdots \times \langle g_{2n}\rangle$ be the abelian group determined by
\begin{equation}
|g_i|=\left\{
       \begin{array}{ll}
         pd, & \hbox{if $i=4k+1$ for some $k$;} \\
         qd, & \hbox{if $i=4k+2$ for some $k$;} \\
         pd^2, & \hbox{if $i=4k+3$ for some $k$;} \\
         qd^2, & \hbox{if $i=4k$ for some $k$.}
       \end{array}
     \right.
\end{equation}
For each $1\leq i\leq n,$ define an element in $\mathbbm{G}$ as follows:
\begin{equation}
h_i=\left\{
      \begin{array}{ll}
        (\mathbbm{g}_{2i-1}\mathbbm{g}_{2i})^{pq}, & \hbox{if $i$ is odd ;} \\
        (\mathbbm{g}_{2i-3}\mathbbm{g}_{2i-2}\mathbbm{g}_{2i-1}\mathbbm{g}_{2i}\mathbbm{g}_{2i+1}\mathbbm{g}_{2i+2})^{pq},
        & \hbox{if $i$ is even and $i\neq n$;} \\
        (\mathbbm{g}_{2i-3}\mathbbm{g}_{2i-2}\mathbbm{g}_{2i-1}\mathbbm{g}_{2i})^{pq}, & \hbox{if $i$ is even and $i=n.$}
     \end{array}
     \right.
\end{equation}
In order to give a datum of Cartan type associated to Cartan matrix $A,$ we need to introduce some characters on $\mathbbm{G}.$
Let $a,b,r$ be the numbers given in Table \ref{tab.2}. Define
\begin{equation}
\chi_1(\mathbbm{g}_j)=\left\{
         \begin{array}{ll}
           \zeta_{d^2}^a, & \hbox{$j=1,2$,} \\
           \zeta_{d^2}^r, & \hbox{$j=3,4$,} \\
           1, & \hbox{$j\geq 5$;}
         \end{array}
       \right.
\end{equation}
\begin{equation}
\chi_2(\mathbbm{g}_j)=\left\{
         \begin{array}{ll}
           \zeta_{d^2}^a, & \hbox{$j=1,2$,} \\
           \zeta_{d^4}, & \hbox{$j=3$,} \\
           \zeta_{d^4}^{-2ad^{2}-1}, & \hbox{$j=4$,} \\
           \zeta_{d^2}^b, & \hbox{$j=5,6$,} \\
           1, & \hbox{$j\geq 7$.}
         \end{array}
       \right.
\end{equation}
For $3\leq i<n,$ define
\begin{equation}
\chi_i(\mathbbm{g}_j)=\left\{
                        \begin{array}{ll}
                          \zeta_{d^2}^{-3b}, & \hbox{if $i$ is odd and $j=2i-3, 2i-2,2i+1, 2i+2$,} \\
                          \zeta_{d^2}^b, & \hbox{if $i$ is odd and $j=2i-1,2i$,} \\
                          \zeta_{d^2}^{b}, & \hbox{if $i$ is even and $j=2i-3, 2i-2,2i+1, 2i+2$,} \\
                          \zeta_{d^4}, & \hbox{if $i$ is even and $j=2i-1$,} \\
                          \zeta_{d^4}^{-2bd^2-1}, & \hbox{if $i$ is even and $j=2i$,} \\
                            1, & \hbox{otherwise.}
                        \end{array}
                      \right.
\end{equation}
When $n$ is odd, define
\begin{equation}
\chi_n(\mathbbm{g}_j)=\left\{
                        \begin{array}{ll}
                          \zeta_{d^2}^{-3b}, & \hbox{$j=2n-3, 2n-2$,} \\
                          \zeta_{d^2}^b, & \hbox{$j=2n-1,2n$,} \\
                            1, & \hbox{otherwise.}
                        \end{array}
                      \right.
\end{equation}
When $n$ is even, define
\begin{equation}
\chi_n(\mathbbm{g}_j)=\left\{
                        \begin{array}{ll}
                          \zeta_{d^2}^{b}, & \hbox{$j=2n-3, 2n-2$,} \\
                          \zeta_{d^4}, & \hbox{$j=2n-1$,} \\
                          \zeta_{d^4}^{-1}, & \hbox{$j=2n$,} \\
                            1, & \hbox{otherwise.}
                        \end{array}
                      \right.
\end{equation}
With these notations, we can give the following lemma.

{\setlength{\unitlength}{1mm}
\begin{table}[t]\centering
\caption{$a,b,r$ associated to $A_n,B_n,C_n.$.} \label{tab.2}
\vspace{1mm}
\begin{tabular}{r|l|l}
\hline
  &\text{$a,b,r$ associated to $A$}& \text{Cartan matrix $A$}\\
\hline
\hline
  1.  & $a=1,\ b=1,\ r=-3 $ &\ \ \ \ \ $A_n$\\
\hline
  2. &$a=2,\ b=1,\ r=-6 $& \ \ \ \ \  $B_n$\\
  \hline
  3. &$a=1,\ b=2,\ r=-4 $& \ \ \ \ \  $C_n$\\
  \hline
 \end{tabular}
\end{table}}
\begin{lemma}
 $\mathfrak{D}=\mathfrak{D}(\mathbbm{G},(h_i)_{1\leq i\leq n},(\chi_i)_{1\leq i\leq n},A)$ is a datum of Cartan type.
\end{lemma}
\begin{proof} Let $q_{ij}=\chi_j(h_i)$ for all $1\leq i\neq j\leq n$.
By definition of datum of Cartan type, we only need to show
\begin{equation}
q_{ij}q_{ji}=q_{ii}^{a_{ij}}=q_{jj}^{a_{ji}}, \ 1\leq i\neq j\leq n,
\end{equation}
and it follows a direct verification.
\end{proof}
\begin{lemma}\label{L5.2}
$\Gamma(\mathfrak{D})$ is not empty.  For each family $\underline{c}=(c_i,c_{jk})_{1\leq i\leq 2n,1\leq j<k\leq 2n}$ in $\Gamma(\mathfrak{D}),$ we have $c_i\neq 0$ for  $1\leq i\leq 2n.$
\end{lemma}
\begin{proof}
Notice that  Equations \eqref{3.8}-\eqref{3.9} always have solutions, for examples $c_{ij}=0$ for all $1\leq i<j\leq 2n.$ So we only need to prove that Equations \eqref{3.7} have solutions and $c_i\neq 0$ for $1\leq i\leq 2n.$ Let $m_i=|g_i|$ for each $1\leq i\leq 2n.$
It is clear that Equations \eqref{3.7} on variables $c_1,c_2,c_3,c_4$ are given by
\begin{eqnarray}
&&c_1ap^2\equiv pq \ \mod pd,\label{5.10} \\
&&c_2aq^2\equiv pq \ \mod qd,\label{5.11} \\
&&c_3p^2\equiv pq \ \mod pd^2, \label{5.12}\\
&&c_3rp^2d^2\equiv 0\ \mod pd^2, \label{5.13}\\
&&c_4(-2ad^2-1)q^2\equiv pq \mod qd^2, \label{5.14}\\
&&c_4rq^2d^2\equiv 0 \mod qd^2.\label{5.15}
\end{eqnarray}
When $3\leq i\leq n$ and $i$ is odd,  Equations \eqref{3.7} are
\begin{eqnarray}
&&c_{2i-1}bp^2\equiv pq \ \mod pd,\label{5.16}\\
&&c_{2i} bq^2\equiv pq \ \mod qd.\label{5.17}
\end{eqnarray}
When $3<i<n$ and $i$ is even,  Equations \eqref{3.7} become
\begin{eqnarray}
&&c_{2i-1}p^2\equiv pq \ \mod pd^2,\label{5.18}\\
&&c_{2i-1}(-3b) p^2d^2\equiv 0 \ \mod pd^2\label{5.19}\\
&&c_{2i}(-2bd^2-1)q^2\equiv pq \ \mod qd^2,\label{5.20}\\
&&c_{2i} (-3b)q^2d^2\equiv 0 \ \mod qd^2.\label{5.21}
\end{eqnarray}
When $i=n$ is even,  the equations  \eqref{3.7} are given by
\begin{eqnarray}
&&c_{2n-1}p^2\equiv pq \ \mod pd^2,\label{5.22}\\
&&c_{2n-1}(-3b) p^2d^2\equiv 0 \ \mod pd^2\label{5.23}\\
&&c_{2n}(-1)q^2\equiv pq \ \mod qd^2,\label{5.24}\\
&&c_{2n} (-3b)q^2d^2\equiv 0 \ \mod qd^2.\label{5.25}
\end{eqnarray}
It is obvious that any integers $c_{2i-1}, c_{2i}$ are solution of \eqref{5.13}, \eqref{5.15}, \eqref{5.19},\eqref{5.21}, \eqref{5.23} and \eqref{5.25}. Since $(ap^2,pd)=p, (aq^2,qd)=q,(p^2,pd^2)=p, ((-2ad^2-1)q^2, qd^2)=q$, so \eqref{5.10}-\eqref{5.15} have solutions by Proposition \eqref{P3.15}. Any solution $c_1, c_2, c_3 \ \mathrm{or}\ c_4$ of \eqref{5.10}-\eqref{5.15} should not be zero since $(p,d)=(q,d)=1$. Similarly, one can show that \eqref{5.16}-\eqref{5.18}, \eqref{5.20},\eqref{5.22} and \eqref{5.24} have nonzero solutions.
\end{proof}
\begin{lemma}\label{L5.3}
There exists modified root vector parameters $\mu$ for $\mathfrak{D}$ satisfying the condition:
\begin{equation}\label{6.25}
   \mu_{\alpha}\ \mathrm{is\ a\ nonzero\ parameter\ if\ and\ only\ if}\ \alpha=\alpha_i\ \mathrm{for\ some\ odd\ number}\ 1\leq i\leq n.
\end{equation}

\end{lemma}
\begin{proof}
Firstly, one can verify that $|\chi_i(h_i)|=d^2$ for $1\leq i\leq n.$ So for all $1\leq i\leq n$ and $i$ is odd, we have $h_i^{d^2}=(\mathbbm{g}_{2i-1}\mathbbm{g}_{2i})^{pqd^2}=g_{2i-1}^{qd}g_{2i}^{pd}\in G,$ $h_i^{d^2}\neq 1$ and $\chi_i^{d^2}=\varepsilon$, hence $\mu_{\alpha_i}\neq 0$.

We proved the lemma.
\end{proof}

\begin{proposition}\label{P5.4}
Let $\underline{c}\in \Gamma(\mathfrak{D})$ and $\mu$ a family of modified root vector parameters for $\mathfrak{D}$ satisfying condition \eqref{6.25}. Then $u(\mathfrak{D},0,\mu,\Phi_{\underline{c}})$ is a finite-dimensional nonradically graded genuine quasi-Hopf algebra associated to $A.$
\end{proposition}
\begin{proof}
By Lemma \ref{L5.2}-\ref{L5.3} and Theorem \ref{T3.4}, $u(\mathfrak{D},0,\mu,\Phi_{\underline{c}})$ is a finite-dimensional nonradically graded quasi-Hopf algebra. So we only need to prove that $u(\mathfrak{D},0,\mu,\Phi_{\underline{c}})$ is a genuine quasi-Hopf algebra.

Let $\alpha$ be a positive root in $R^+$, then we have $\mu_{\alpha}=0$ if $\alpha\neq \alpha_i$ some odd number $1\leq i\leq n.$ Let $\mathcal{I}$ be the ideal of $u(\mathfrak{D},0,\mu,\Phi_{\underline{c}})$ generated by $$\{X_i, 1-g_j|1\leq i\leq n, j=1\ \mathrm{or}\ 2\ \mod 4\}.$$  It is obvious that $\mathcal{I}$ is a quasi-Hopf ideal of $u(\mathfrak{D},0,\mu,\Phi_{\underline{c}}).$ Denote by
$$G'=\langle g_i|i=0\ \mathrm{or}\ 3\ \mod 4\rangle.$$
Then it is obvious that
\begin{equation}
\k G'=u(\mathfrak{D},0,\mu,\Phi_{\underline{c}})/\mathcal{I}.
\end{equation}
Similar to the proof of Proposition \ref{P4.6}, one can show that the associator of the $\k G'$ is $\overline{\Phi_{\underline{c}}}=\sum_{e,f,g\in G'}\phi_{\underline{c}}|_{G'}(e,f,g)1_e\otimes 1_f\otimes 1_g.$ By lemma \ref{L5.2}, $\phi_{\underline{c}}|_{G'}$ is a not $3$-coboundary on $G'.$ Hence $(\k G',\overline{\Phi_{\underline{c}}})$ is a genuine quasi-Hopf algebra. This implies $u(\mathfrak{D},0,\mu,\Phi_{\underline{c}})$ is genuine, since otherwise $(\k G',\overline{\Phi_{\underline{c}}})=u(\mathfrak{D},0,\mu,\Phi_{\underline{c}})/\mathcal{I}$ should not be genuine, which is a contradiction.

\end{proof}

\subsection{Quasi-Hopf algebras of Cartan type $D_n$} In this subsection, we always assume that $n\geq 4.$
 Let $G=\langle g_1\rangle\times\cdots \times \langle g_{2n}\rangle$ such that
\begin{equation}
|g_i|=\left\{
        \begin{array}{ll}
          pd, & \hbox{$i=1\ \mathrm{or}\ 4k-1\ \mathrm{for}\ k\geq 1$;} \\
          qd, & \hbox{$i=2\ \mathrm{or}\ 4k\ \mathrm{for}\ k\geq 1$;} \\
          pd^2, & \hbox{$i=4k+1\ \mathrm{for}\ k\geq 1$;} \\
          qd^2, & \hbox{$i=4k+2\ \mathrm{for}\ k\geq 1$.}
        \end{array}
      \right.
\end{equation}
Let $(h_i)_{1\leq i\leq n}$ be a family of elements in $\mathbbm{G}$ given by
\begin{equation}
h_i=\left\{
      \begin{array}{ll}
        (\mathbbm{g}_{2i-1}\mathbbm{g}_{2i})^{pq}, & \hbox{$\mathrm{if}\ i=1\ \mathrm{or}\ 2k\ \mathrm{for}\ k\geq 1$;} \\
        (\prod_{j=1}^8\mathbbm{g}_j)^{pq}, & \hbox{$\mathrm{if}\ i=3$;} \\
        (\mathbbm{g}_{2i-3}\mathbbm{g}_{2i-2}\mathbbm{g}_{2i-1}\mathbbm{g}_{2i}\mathbbm{g}_{2i+1}\mathbbm{g}_{2i+2})^{pq}, &
         \hbox{$\mathrm{if}\ i=2k+1\neq n\ \mathrm{for}\ k\geq 2$;} \\
        (\mathbbm{g}_{2n-3}\mathbbm{g}_{2n-2}\mathbbm{g}_{2n-1}\mathbbm{g}_{2n})^{pq}, & \hbox{$\mathrm{if}\ i=n\ \mathrm{is}\ odd$.}
       \end{array}
      \right.
\end{equation}
Now define a family $(\chi_i)_{1\leq i\leq n}$ of characters on $\mathbbm{G}$ as following.

When $i=1,2,$
\begin{equation}
\chi_i(g_j)=\left\{
              \begin{array}{ll}
                \zeta_{d^2}, & \hbox{$j=2i-1,2i$;} \\
                \zeta_{d^2}^{-3}, & \hbox{$j=5,6$;} \\
                1, & \hbox{otherwise.}
              \end{array}
            \right.
\end{equation}
\begin{equation}
\chi_3(g_j)=\left\{
              \begin{array}{ll}
                \zeta_{d^2}, & \hbox{$j=1,2,3,4,7,8$;} \\
                \zeta_{d^4}, & \hbox{$j=5$;} \\
                \zeta_{d^4}^{-4d^2-1}, & \hbox{$j=6$;} \\
                1, & \hbox{otherwise.}
              \end{array}
            \right.
\end{equation}
When $3<i<n$,
\begin{equation}
\chi_i(\mathbbm{g}_j)=\left\{
              \begin{array}{ll}
                \zeta_{d^2}^{-3}, & \hbox{if $i$ is even and $j=2i-3,2i-2,2i+1,2i+2$ ;} \\
                \zeta_{d^2}, & \hbox{if $i$ is even and $j=2i-1,2i$ ;} \\
                \zeta_{d^4}, & \hbox{if $i$ is odd and $j=2i-1$;} \\
                \zeta_{d^4}^{-2d^2-1}, & \hbox{if $i$ is odd and $j=2i$;} \\
                \zeta_{d^2}, & \hbox{if $i$ is odd and $j=2i-3,2i-2,2i+1,2i+2$;} \\
                1, & \hbox{otherwise.}
              \end{array}
            \right.
\end{equation}
If $n$ is even,
\begin{equation}
\chi_n(\mathbbm{g}_j)=\left\{
              \begin{array}{ll}
                \zeta_{d^2}, & \hbox{if $j=2n-1,2n$ ;} \\
                \zeta_{d^2}^{-3}, & \hbox{if $j=2n-3,2n-2$ ;} \\
                1, & \hbox{otherwise.}
              \end{array}
            \right.
\end{equation}
If $n$ is odd,
\begin{equation}
\chi_n(\mathbbm{g}_j)=\left\{
              \begin{array}{ll}
                \zeta_{d^2}, & \hbox{if $j=2n-3,2n-2$ ;} \\
                \zeta_{d^4}, & \hbox{if $j=2n-1$ ;} \\
                \zeta_{d^4}^{-1}, & \hbox{if $j=2n$ ;} \\
                1, & \hbox{otherwise;}
              \end{array}
            \right.
\end{equation}
With these definitions, one can verify that $\mathfrak{D}= \mathfrak{D}(\mathbbm{G},(h_i)_{1\leq i\leq n},(\chi_i)_{1\leq i\leq n}, D_n)$ is a datum of Cartan type. Moreover, we have the following two lemmas, and the proofs, omitted, are similar as the proofs of Lemma \ref{L5.2} and \ref{L5.3}.

\begin{lemma}\label{L5.5}
$\Gamma(\mathfrak{D})$ is a nonempty set, and for each family $\underline{c}=(c_i,c_{jk})_{1\leq i\leq 2n,1\leq j<k\leq 2n}$ in $\Gamma(\mathfrak{D}),$ we have $c_i\neq 0$ for  $1\leq i\leq 2n.$
\end{lemma}

\begin{lemma}\label{L5.6}
There exists a family of modified root vector parameters $\mu$ for $\mathfrak{D}$ satisfying the condition:
\begin{equation}\label{6.34}
   \mu_{\alpha}\ \mathrm{is\ a\ nonzero\ if\ and\ only\ if}\ \alpha=\alpha_1 \ \mathrm{or}\ \alpha_i\ \mathrm{for\ some\ even\ number}\ 1\leq i\leq n.
\end{equation}
 \end{lemma}

\begin{proposition}\label{P5.7}Let $\underline{c}\in \Gamma(\mathfrak{D}),$ and $\mu$ a family of modified root vector parameters for $\mathfrak{D}$ satisfying the condition \eqref{6.34}.  Then $u(\mathfrak{D},0,\mu,\Phi_{\underline{c}})$ is a genuine quasi-Hopf algebra.
\end{proposition}
\begin{proof}
Similar to the proof of Proposition \ref{P5.4}.
\end{proof}

\subsection{Quasi-Hopf algebras of Cartan type $E_6$, $E_7$ and $E_8$ } In this subsection, we always assume $n=6,7\ \mathrm{or}\ 8$.
Define an abelian group $G_8=\langle g_1\rangle \times \langle g_2\rangle \times \cdots \times \langle g_{16}\rangle$ such that
\begin{eqnarray}
&&|g_{2i-1}|=pd^2,\ |g_{2i}|=qd^2\ \mathrm{for}\ i=1,3,6,8\\
&&|g_{2i-1}|=pd,\ |g_{2i}|=qd\ \mathrm{for}\ i=2,4,5,7.
\end{eqnarray}
Let $G_6=\langle g_1\rangle \times \langle g_2\rangle \times \cdots \times \langle g_{12}\rangle$ and $G_7=\langle g_1\rangle \times \langle g_2\rangle \times \cdots \times \langle g_{14}\rangle$.
Define
\begin{eqnarray}
&&h_i=(\mathbbm{g}_{2i-1}\mathbbm{g}_{2i})^{pq}\ \mathrm{for}\  i = 2,4,5,7,\\
&&h_1=(\mathbbm{g}_1\mathbbm{g}_2\mathbbm{g}_3\mathbbm{g}_4)^{pq},\\
&&h_3=(\prod_{i=3}^{10}\mathbbm{g}_i)^{pq},\\
&&h_6=(\mathbbm{g}_9\mathbbm{g}_{10}\mathbbm{g}_{11}\mathbbm{g}_{12})^{pq},\\
&&h'_6=(\mathbbm{g}_9\mathbbm{g}_{10}\mathbbm{g}_{11}\mathbbm{g}_{12}\mathbbm{g}_{13}\mathbbm{g}_{14})^{pq},\\
&&h_8=(\mathbbm{g}_{13}\mathbbm{g}_{14}\mathbbm{g}_{15}\mathbbm{g}_{16})^{pq}.
\end{eqnarray}

{\setlength{\unitlength}{1mm}
\begin{table}[t]\centering
\caption{Characters of $\mathbbm{G}_8=\langle \mathbbm{g}_1\rangle \times \cdots \times \langle \mathbbm{g}_{16}\rangle$ } \label{tab.3}
\vspace{1mm}
\begin{tabular}{r|l|l|l|l|l|l|l|l|l}
\hline
  &$\chi_1$& $\chi_2$&$\chi_3$&$\chi_4$&$\chi_5$&$\chi_6$&$\chi'_6$&$\chi_7$&$\chi_8$\\
\hline
\hline
 $\mathbbm{g}_1$ & $\zeta_{d^4}$ & $\zeta_{d^2}^{-3}$ &1 &1&1&1&1&1&1\\
\hline
  $\mathbbm{g}_2$ & $\zeta_{d^4}^{-1}$ &$\zeta_{d^2}^{-3}$ &1 &1&1&1&1&1&1\\
\hline
   $\mathbbm{g}_3$ & $\zeta_{d^2}$ &$\zeta_{d^2}$ &$\zeta_{d^2}$ &1&1&1&1&1&1\\
\hline
 $\mathbbm{g}_4$ & $\zeta_{d^2}$ &$\zeta_{d^2}$ &$\zeta_{d^2}$ &1&1&1&1&1&1\\
\hline
 $\mathbbm{g}_5$ &1&$\zeta_{d^2}^{-3}$ &$\zeta_{d^4}$ &$\zeta_{d^2}^{-3}$ &$\zeta_{d^2}^{-3}$&1&1&1&1\\
\hline
$\mathbbm{g}_6$ &1&$\zeta_{d^2}^{-3}$ &$\zeta_{d^4}^{-4d^2-1}$ &$\zeta_{d^2}^{-3}$ &$\zeta_{d^2}^{-3}$&1&1&1&1\\
\hline
$\mathbbm{g}_7$ &1&1&$\zeta_{d^2}$ &$\zeta_{d^2}$ &1&1&1&1&1\\
\hline
$\mathbbm{g}_8$ &1&1&$\zeta_{d^2}$ &$\zeta_{d^2}$ &1&1&1&1&1\\
\hline
$\mathbbm{g}_9$ &1&1&$\zeta_{d^2}$ &1 &$\zeta_{d^2}$&$\zeta_{d^2}$&$\zeta_{d^2}$&1&1\\
\hline
$\mathbbm{g}_{10}$ &1&1&$\zeta_{d^2}$ &1 &$\zeta_{d^2}$&$\zeta_{d^2}$&$\zeta_{d^2}$&1&1\\
\hline
$\mathbbm{g}_{11}$ &1&1&1&1 &$\zeta_{d^2}^{-3}$&$\zeta_{d^4}$&$\zeta_{d^4}$&$\zeta_{d^2}^{-3}$&1\\
\hline
$\mathbbm{g}_{12}$ &1&1&1&1 &$\zeta_{d^2}^{-3}$&$\zeta_{d^4}^{-1}$&$\zeta_{d^4}^{-2d^2-1}$&$\zeta_{d^2}^{-3}$&1\\
\hline
$\mathbbm{g}_{13}$ &1&1&1&1&1&1&$\zeta_{d^2}$&$\zeta_{d^2}$&$\zeta_{d^2}$\\
\hline
$\mathbbm{g}_{14}$ &1&1&1&1&1&1&$\zeta_{d^2}$&$\zeta_{d^2}$&$\zeta_{d^2}$\\
\hline
$\mathbbm{g}_{15}$ &1&1&1&1&1&1&1&$\zeta_{d^2}^{-3}$&$\zeta_{d^4}$\\
\hline
$\mathbbm{g}_{16}$ &1&1&1&1&1&1&1&$\zeta_{d^2}^{-3}$&$\zeta_{d^4}^{-1}$\\
\hline
 \end{tabular}
\end{table}}
Let $\chi_i,1\leq i\leq 8$ and $\chi'_6$ be the characters of $\mathbbm{G}$ given in Table \ref{tab.3}, then one can verify that
\begin{eqnarray*}
&&\mathfrak{D}_6=\mathfrak{D}(\mathbbm{G}_6,(h_1,\cdots,h_6),(\chi_1,\cdots,\chi_6),E_6),\\ &&\mathfrak{D}_7=\mathfrak{D}(\mathbbm{G}_7,(h_1,\cdots,h_5,h'_6,h_7),(\chi_1,\cdots,\chi_5,\chi'_6,\chi_7),E_7),\\
&&\mathfrak{D}_8=\mathfrak{D}(\mathbbm{G}_8,(h_1,\cdots,h_5,h'_6,h_7,h_8),(\chi_1,\cdots,\chi_5,\chi'_6,\chi_7,\chi_8),E_8).
\end{eqnarray*} are datums of Cartan type.  And similar as Lemma \ref{L5.2} and \ref{L5.3} we have the following:
\begin{lemma}
$\Gamma(\mathfrak{D}_n)$ is a nonempty set.  For  any $\underline{c}=(c_i,c_{jk})_{1\leq i\leq n, 1\leq j<k\leq n}$  in $\Gamma(\mathfrak{D}_n),$  we have  $c_i\neq 0$ for each $1\leq i\leq 2n.$
\end{lemma}

\begin{lemma}\label{L5.9}
There exists nonzero modified root vector parameters $\mu$ for $\mathfrak{D}_n$ satisfying the conditions: if $n=6,$ $\mu_{\alpha}\neq 0$ if and only if $\alpha=\alpha_i$ for $i=2,4\ \mathrm{or}\ 5$; if $n=7\ \mathrm{or}\ 8,$ $\mu_{\alpha}\neq 0$ if and only if $\alpha=\alpha_i$ for $i=2,4,5\ \mathrm{or}\ 7$.
\end{lemma}

\begin{proposition}\label{P5.10}
Let $\underline{c}\in \Gamma(\mathfrak{D}_n),$ and $\mu$ a family of modified root vector parameters for $\Gamma(\mathfrak{D}_n)$ satisfying the conditions of Lemma \ref{L5.9}.  Then $u(\mathfrak{D}_n,0,\mu,\Phi_{\underline{c}})$ is a nonradically graded genuine quasi-Hopf algebra.
\end{proposition}
\begin{proof}
Similar to the proof of Proposition \ref{P5.4}.
\end{proof}

\subsection{Quasi-Hopf algebras of Cartan type $F_4$}
 Let $G=\langle g_1\rangle \times \cdots\times \langle g_8\rangle$ such that
\begin{equation}
|g_i|=\left\{
        \begin{array}{ll}
          pd, & \hbox{$i=1,5$;} \\
          qd, & \hbox{$i=2,6$;} \\
          pd^2, & \hbox{$i=3,7$;} \\
          qd^2, & \hbox{$i=4,8$.}
        \end{array}
      \right.
\end{equation}
Denote by $(h_i)_{1\leq i\leq 4}$ a family of elements in $\mathbbm{G}$ through
\begin{eqnarray}
&&h_1=(\mathbbm{g}_1\mathbbm{g}_2)^{pq},\ \  h_2=(\mathbbm{g}_1\mathbbm{g}_2\mathbbm{g}_3\mathbbm{g}_4\mathbbm{g}_5\mathbbm{g}_6)^{pq},\\
&&h_3=(\mathbbm{g}_5\mathbbm{g}_6)^{pq},\ \  h_4=(\mathbbm{g}_5\mathbbm{g}_6\mathbbm{g}_7\mathbbm{g}_8)^{pq}.\notag
\end{eqnarray}

{\setlength{\unitlength}{1mm}
\begin{table}[t]\centering
\caption{Characters of $\mathbbm{G}=\langle \mathbbm{g}_1\rangle \times \cdots \times \langle \mathbbm{g}_8\rangle $ } \label{tab.4}
\vspace{1mm}
\begin{tabular}{r|l|l|l|l|l|l|l|l}
\hline
  &$\mathbbm{g}_1$& $\mathbbm{g}_2$&$\mathbbm{g}_3$&$\mathbbm{g}_4$&$\mathbbm{g}_5$&$\mathbbm{g}_6$&$\mathbbm{g}_7$&$\mathbbm{g}_8$\\
\hline
\hline
 $\chi_1$ & $\zeta_{d^2}$ & $\zeta_{d^2}$ &$\zeta_{d^2}^{-3}$ &$\zeta_{d^2}^{-3}$ &1&1&1&1\\
\hline
  $\chi_2$ & $\zeta_{d^2}$& $\zeta_{d^2}$ &$\zeta_{d^4}$&$\zeta_{d^4}^{-4d^2-1}$&$\zeta_{d^2}$&$\zeta_{d^2}$&1&1\\
\hline
   $\chi_3$ & 1&1&$\zeta_{d^2}^{-6}$ &$\zeta_{d^2}^{-6}$&$\zeta_{d^2}^2$&$\zeta_{d^2}^2$&$\zeta_{d^2}^{-6}$&$\zeta_{d^2}^{-6}$\\
\hline
 $\chi_4$ & 1 &1&1&1&$\zeta_{d^2}^2$&$\zeta_{d^2}^2$&$\zeta_{d^4}$&$\zeta_{d^4}^{-1}$\\
\hline
 \end{tabular}
\end{table}}
Let $(\chi_i)_{1\leq i\leq 4}$ be the characters of $\mathbbm{G}$ given in Table \ref{tab.4}, then one can easily verify that $$\mathfrak{D}=\mathfrak{D}(\mathbbm{G},(h_i)_{1\leq i\leq 4},(\chi_i)_{1\leq i\leq 4},F_4)$$ is a datum of Cartan type. Similar to Lemma \ref{L5.2}-\ref{L5.3}, we have the following two lemmas.

\begin{lemma}\label{L5.10}
$\Gamma(\mathfrak{D})$ is a nonempty set. Let $\underline{c}=(c_i,c_{jk})_{1\leq i\leq 8, 1\leq j<k\leq 8}$ be an element in $\Gamma(\mathfrak{D})$.  Then we have $c_i\neq 0$ for $1\leq i\leq 8.$
\end{lemma}

\begin{lemma}\label{L5.11}
There exists a family of modified root vector parameter $\mu$ for $\Gamma(\mathfrak{D})$ satisfying the condition: $\mu_{\alpha}$ is nonzero  if and only if  $\alpha=\alpha_1\ \mathrm{or}\ \alpha_3$.
\end{lemma}

 \begin{proposition}\label{P5.13}
Let $\underline{c}\in \Gamma(\mathfrak{D}),$  and $\mu$ a family of modified root vector parameters satisfying the condition of  Lemma \ref{L5.11}. Then $u(\mathfrak{D},0,\mu,\Phi_{\underline{c}})$ is a nonradically graded genuine quasi-Hopf algebra.
\end{proposition}
\begin{proof}
Similar to the proof of Proposition \ref{P5.4}.
\end{proof}

\end{document}